\documentclass[article,onesided]{amsart}\def\loadTIKZ{\usepackage{tikz}\usetikzlibrary{matrix,arrows,calc,cd,decorations.pathmorphing}}\ifdefined\headpresent\else \def\headpresent{}\overfullrule=5pt \newcommand\hfuzzReset{\hfuzz=3pt}\hfuzzReset \newcommand\toleranceReset{\tolerance=1400}\toleranceReset \newcommand\emergencystretchReset{\emergencystretch=1ex}\emergencystretchReset \hbadness=10000 \usepackage{ifpdf}\newcommand{\addressTressl}{The University of Manchester, School of Mathematics, Oxford Road, Manchester M13 9PL, UK}\newcommand{\emailTressl}{marcus.tressl@manchester.ac.uk}\newcommand{\homepageTressl}{\url{http://personalpages.manchester.ac.uk/staff/Marcus.Tressl/index.php}}\let\oldtocsection=\tocsection \let\oldtocsubsection=\tocsubsection \let\oldtocsubsubsection=\tocsubsubsection \renewcommand{\tocsection}[2]{\hspace{0em}\vspace*{0.1em}\oldtocsection{#1}{#2}}\renewcommand{\tocsubsection}[2]{\hspace{4ex}\oldtocsubsection{#1}{#2}}\renewcommand{\tocsubsubsection}[2]{\hspace{6ex}\oldtocsubsubsection{#1}{#2}}\ifpdf \usepackage[pdftex]{lscape}\else \usepackage{lscape}\fi \usepackage{ulem}\usepackage{fancybox}\usepackage{xifthen}\usepackage{forarray}\usepackage{xstring}\usepackage{stringstrings}\def\StackCreate#1#2#3{  \expandafter\def\csname#1\endcsname{#3}  \expandafter\def\csname#1Push\endcsname##1{\expandafter\edef\csname#1\endcsname{##1#2\csname#1\endcsname}}  \expandafter\def\csname TopAux#1\endcsname ##1#2##2#3{##1}  \expandafter\def\csname#1Top\endcsname{\expandafter\expandafter\expandafter\expandafter\expandafter\expandafter\csname TopAux#1\endcsname\csname#1\endcsname}  \expandafter\def\csname PopAux#1\endcsname ##1#2##2#3##3#2{\expandafter\def\csname##3\endcsname{##2#3}}  \expandafter\def\csname#1Pop\endcsname{\expandafter\expandafter\expandafter\expandafter\expandafter\expandafter\csname PopAux#1\endcsname\csname#1\endcsname#1#2}}\def\GetAfterColonAux#1:#2;{#2}\def\GetAfterColon#1{\IfBeginWith{#1}{:}{\GetAfterColonAux#1;}{#1}}\usepackage[shortlabels,inline]{enumitem}\setenumerate[1]{leftmargin=5.5ex}\setitemize[1]{leftmargin=5.5ex}\SetEnumitemKey{noindent}{leftmargin=0ex, itemindent=5ex, align=right, itemsep=1ex }\newcommand\NOPAGENUMBER[1]{}\usepackage{everypage}\newcommand\AddPrivateToMargin[1]{\AddEverypageHook{\tikz[overlay,remember picture]{    \node at ($(current page.west)+(1.5,0)$) [rotate=90] {\textcolor{orange}{\vbox{\hrule width \the\textwidth height 0.5pt} \textcolor{black}{#1}\ \vbox{\hrule width 40em height 0.5pt}}}; }}}\newcommand\AddLongversionToMargin[1]{\AddEverypageHook{\tikz[overlay,remember picture]{    \node at ($(current page.west)+(2,0)$) [rotate=90] {\textcolor{\LongColor}{\vbox{\hrule width  \the\textwidth height 0.5pt} \textcolor{black}{#1}\ \vbox{\hrule width 40em height 0.5pt}}}; }}}\newcommand\AddOldversionToMargin[1]{\AddEverypageHook{\tikz[overlay,remember picture]{    \node at ($(current page.west)+(2.5,0)$) [rotate=90] {\textcolor{\OldColor}{\vbox{\hrule width \the\textwidth height 0.5pt} \textcolor{black}{#1}\ \vbox{\hrule width 40em height 0.5pt}}}; }}}\IfFileExists{mathabx.sty}{}{}\usepackage{amsfonts}\usepackage{amssymb}\usepackage{stmaryrd}\usepackage{amsmath}\usepackage{amsthm}\usepackage{dsfont}\IfFileExists{mbboard.sty}{\usepackage{mbboard}}{}\usepackage{mathrsfs}\usepackage{twcal}\usepackage[latin1]{inputenc}\usepackage{accents}\usepackage{bm}\ifpdf \usepackage[pdftex,usenames,x11names]{xcolor}\else \usepackage[dvips,usenames,x11names]{xcolor}\fi \usepackage[pdftex]{graphicx}\usepackage[all]{xy}\ifdefined\loadTIKZ   \loadTIKZ \def\TIKZlabel#1{}\else\fi \StackCreate{ColoR}{;}{?}\ColoRPush{black}  \newcommand {\notion}[2][]{  \def\temp{#1}  \ifmmode     #2     \ifx \temp\empty     \index{$#2$}    \else     \index{$#1$}    \fi   \else {\bf #2}    \ifx \temp\empty     \index{#2}    \else     \index{#1}    \fi   \fi   }\ifpdf \usepackage[pdftex,linktocpage,pagebackref,breaklinks]{hyperref}\else \usepackage[hypertex,linktocpage,pagebackref]{hyperref}\fi \usepackage{xr-hyper}\newcommand{\refX}[2]{\IfBeginWith{#1}{:}{\ref{\GetAfterColonAux#1;-#2}}{\cite[\ref{#1-#2}]{#1}}}\fi \newcommand\pr{\begin{proof}}\def\ende{\end{proof}}\newtheoremstyle{LayoutVoid}  {1ex}  {0ex}  {\normalfont}  {}  {\bf}  {.}  {1ex}  {}\newcommand\stressstatement[1]{#1}\theoremstyle{plain}\swapnumbers \newtheorem{theorem}{Theorem}[section]\newcommand\maketheorem[1]{   \newtheorem{#1}[theorem]{\stressstatement{#1}}    \newtheorem{#1Definition}[theorem]{\stressstatement{#1 and Definition}}   }\FunctionForEach{,}{\maketheorem}{Conclusion,Conjecture,Corollary,Fact,Facts,Lemma,Observation,Observations,Proposition,Reminder,Scholium,Summary,Theorem}\theoremstyle{definition}\theoremstyle{remark}\FunctionForEach{,}{\maketheorem}{Convention,Counterexample,Discussion,Example,Examples, Exercise,Exercises,Explanation,Notation,Project,Projects,Question,Questions,Remark,Remarks,Strategy,Warning}\theoremstyle{LayoutVoid}\numberwithin{equation}{section}\newcommand{\labelon}[1]{\marginpar{#1}}\newcommand{\labelx}[1]{{\def\temp{#1}\ifx\temp\empty\else \label{#1}\labelon{#1}\fi}}\def\GetAfterColon#1:#2;;{#2}\def\GetAfterPlus#1+#2;;{#2}\newenvironment{FACT}[2]    {     \IfBeginWith{#1}{:}        {\def\tempFactName{void}\def\tempFreeTitle{\GetAfterColon#1;;\ }}        {            \IfBeginWith{#1}{+}              {\def\tempFactName{voidTheorem}\def\tempFreeTitle{\GetAfterPlus#1;;\ }}              {\def\tempFactName{#1}\def\tempFreeTitle{}}        }     \def\tempfacT{\end{\tempFactName}}    \begin{\tempFactName}\labelx{#2}\textup{\textbf{\tempFreeTitle}}    \capitalize[q]{#1}        \caselower[q]{#1}    \global\edef\factname{\thestring}    }    {\tempfacT}\catcode`\=13 \def{+}\newcommand\assigncharacter[1]{\expandafter\newcommand\csname #1\endcsname{\mathds{#1}}}\FunctionForEach{,}{\assigncharacter}{A,B,C,D,E,F,G,I,J,K,M,N,Q,R,T,U,V,W,X,Y,Z}\renewcommand\assigncharacter[1]{\expandafter\newcommand\csname C#1\endcsname{\mathcal{#1}}}\FunctionForEach{,}{\assigncharacter}{A,B,C,D,E,F,G,H,I,J,K,L,M,N,O,P,Q,R,S,T,U,V,W,X,Y,Z}\renewcommand\assigncharacter[1]{\expandafter\newcommand\csname D#1\endcsname{\mathfrak{#1}}}\FunctionForEach{,}{\assigncharacter}{a,b,c,d,e,f,g,h,i,j,k,l,m,n,o,p,q,r,s,t,u,v,w,x,y,z,A,B,C,D,E,F,G,I,K,L,M,N,O,P,Q,R,S,T,U,V,W,X,Y,Z} \renewcommand{\SS}{\mathscr{S}}\renewcommand\assigncharacter[1]{\expandafter\newcommand\csname S#1\endcsname{\mathscr{#1}}}\FunctionForEach{,}{\assigncharacter}{A,B,C,D,E,F,G,H,I,J,K,L,M,N,O,P,Q,R,T,U,V,W,X,Y,Z}\def\NewFont#1#2#3#4#5{\expandafter\font\csname #1display\endcsname =#1 at #2 \expandafter\font\csname #1normal\endcsname =#1 at #3 \expandafter\font\csname #1script\endcsname =#1 at #4 \expandafter\font\csname #1scriptscript\endcsname =#1 at #5 }\def\NewFontLetter#1#2{{\mathchoice {{\expandafter\hbox{\csname #1display\endcsname\char"#2}}}{{\expandafter\hbox{\csname #1normal\endcsname\char"#2}}}{{\expandafter\hbox{\csname #1script\endcsname\char"#2}}}{{\expandafter\hbox{\csname #1scriptscript\endcsname\char"#2}}}}}\NewFont{pxsyc}{9.00pt}{8.00pt}{7.00pt}{6.00pt}\NewFont{pxsya}{9.00pt}{8.00pt}{7.00pt}{6.00pt}\NewFont{p1xr}{10.00pt}{9.00pt}{8.00pt}{7.00pt}\NewFont{MnSymbolC10}{10.00pt}{9.00pt}{8.00pt}{7.00pt}\NewFont{MnSymbolD10}{12.00pt}{11.00pt}{10.00pt}{9.00pt}\NewFont{MnSymbolF10}{12.00pt}{11.00pt}{10.00pt}{9.00pt}\NewFont{manfnt}{12.00pt}{11.00pt}{10.00pt}{9.00pt}\NewFont{favmr7y}{12.00pt}{11.00pt}{10.00pt}{9.00pt}\newcommand\bdl{{\ifmmode \mathrm{bdlat}\else {bounded distributive lattice}\fi}} \newcommand{\st}{{\ \vert \ }}\let\temp\phi \let\phi\varphi \let \varphi\temp \let\temp\theta \let\theta\vartheta \let \vartheta\temp \let\eps\varepsilon  \let\0\emptyset \newcommand{\into}{\hookrightarrow}\newcommand{\onto}{\twoheadrightarrow}\newcommand{\lra}{\longrightarrow}\newcommand{\Ra}{\Rightarrow}\newcommand{\alg}{\mathrm{alg}}\newcommand{\mal}{\cdot}\newcommand{\monthname}[1]{\ifcase#1 \or January \or February \or March \or April \or May \or June \or July \or August \or September \or October \or November \or December \fi}\newcommand\LongColor{teal}\newcommand\OldColor{gray}\newcommand\LongStart{\ColoRPush{\LongColor}\color{\ColoRTop}[BEGIN LONG VERSION]}\newcommand\LongEnd{[END LONG VERSION]\ColoRPop\color{\ColoRTop}}\newcommand\OldStart{\ColoRPush{\OldColor}\color{\ColoRTop}[BEGIN OLD VERSION]}\newcommand\OldEnd{[END OLD VERSION]\ColoRPop\color{\ColoRTop}}\newcommand\kat[1]{{\tt #1}}\newcommand{\claim}{\underline{Claim.}\ }\newcommand{\Claim}[1]{\underline{Claim #1.}\ }\newcommand\map[5]{\begin{eqnarray*}  #1#2&\lra &#3 \\   #4&\longmapsto &#5 \end{eqnarray*}} \renewcommand{\mod}{{\operator{\,mod\,}}}\newcommand\operator[1]{\mathop{\operatorname{#1}}\nolimits} \newcommand{\id}{\operator{id}}\newcommand{\Der}{\operator{Der}}\newcommand{\CODF}{{\rm CODF}}\newcommand{\Th}{\operator{Th}}\newcommand{\card}{\operator{card}}\newcommand{\qf}{\operator{qf}}\newcommand{\Hom}{\operator{Hom}}\newcommand{\trdeg}{\operator{tr.deg}}\newcommand{\Jet}{\operator{Jet}}\newcommand\DIRroot{../../../../../../../../../../../}\definecolor{Green}{rgb}{0.00,0.50,0.00}\definecolor{grey}{rgb}{0.40,0.40,0.40} \renewcommand\textcolor[2]{\ColoRPush{#1}\color{\ColoRTop}#2\ColoRPop\color{\ColoRTop}}\IfFileExists{C:/wb/System64/WinBatch.exe}{}{}\ifdefined\isinput\endinput\else\fi \let\temp\theta \let\theta\vartheta \let \vartheta\temp \def\uc{{\rm UC} }\def\Ind{\setbox0=\hbox{$x$}\kern\wd0\hbox to 0pt{\hss$\mid$\hss} \lower.9\ht0\hbox to 0pt{\hss$\smile$\hss}\kern\wd0}\def\Notind{\setbox0=\hbox{$x$}\kern\wd0\hbox to 0pt{\mathchardef \nn=12854\hss$\nn$\kern1.4\wd0\hss}\hbox to 0pt{\hss$\mid$\hss}\lower.9\ht0 \hbox to 0pt{\hss$\smile$\hss}\kern\wd0}\def \U {\mathcal U}\def \DCF {\operatorname{DCF}}\def \UC {\operatorname{UC}}\def \jet {\operatorname{Jet}}\def \alg {\operatorname{alg}}\def \pol {\operatorname{pol}}\def \SS {\mathcal S}\def \V{\mathcal V}\renewcommand{\labelon}[1]{}\newif\ifprivate\privatefalse \newif\iflongversion\longversionfalse \newif\ifoldversion\oldversionfalse \renewcommand\LongStart{\ColoRPush{\LongColor}\color{\ColoRTop}}\renewcommand\LongEnd{\ColoRPop\color{\ColoRTop}}\ifprivate \AddPrivateToMargin{private version} \fi \iflongversion \AddLongversionToMargin{long version} \fi \ifoldversion \AddOldversionToMargin{old version included} \fi 
\begin{document}\title{Differential Weil Descent and Differentially Large Fields}\author{Omar Le\'on S\'anchez}\address{\addressTressl}\email{omar.sanchez@manchester.ac.uk}\author{Marcus Tressl}\address{\addressTressl\newline Homepage: \homepageTressl}\email{\emailTressl}\date{\today}\begin{abstract}A differential version of the classical Weil descent is established in all characteristics. It yields a theory of differential restriction of scalars for differential varieties over finite differential field extensions. This theory is then used to prove that in characteristic 0, \textit{differential largeness} (a notion introduced here as an analogue to largeness of fields) is preserved under algebraic extensions. This provides many new differential fields with minimal differential closures. A further application is Kolchin-density of rational points in differential algebraic groups defined over differentially large fields. \end{abstract}\maketitle \tableofcontents \section{Introduction}\noindent In 1959, Andr\'e Weil introduced his method of restricting scalars for a finite separable field extension $K\subseteq L$, cf. \cite[\S1.3]{Weil1982}. It says that scalar extension, seen as a functor from $K$-algebras to $L$-algebras, has a left adjoint, which sends an $L$-algebra $D$ to a $K$-algebra $W(D)$, the Weil descent (aka Weil restriction) of $D$ from $L$ to $K$. The construction has been vastly generalised by Grothendieck, see \cite{Grothe1995a}. \ \par We establish a similar descent for differential algebras with respect to a given extension of differential rings $A\subseteq B$, where $B$ is finitely generated and free as an $A$-module. Here a differential ring $A$ is a commutative unital ring equipped with a finite set of derivations $A\lra A$. If $D$ is a differential $B$-algebra with commuting derivations, its descent $W^\mathrm{diff}(D)$ is a differential $A$-algebra in commuting derivations, see \ref{DiffWeilDescent}. This is deduced from our first main result, which concerns rings and algebras with a single derivation: \ \par \medskip\noindent \textbf{Theorem A.} (see \ref{partialDW} and \ref{liebracket})\ \par \noindent Let $d:A\lra A$ be a derivation of a ring $A$ and let $(B,\delta)$ be a differential $(A,d)$-algebra. Assume that $B$ is finitely generated and free as an $A$-module. \begin{enumerate}[(i),itemsep=1ex]\item Let $(D,\partial)$ be a differential $(B,\delta)$-algebra. Then there is a unique derivation $\partial^W$ on the classical Weil descent $W(D)$ such that $(W(D),\partial^W)$ is a differential $(A,d)$-algebra and the unit of the adjunction at $D$ (given by the classical Weil descent), namely the map $W_D:D\lra W(D)\otimes B$, is a differential $(B,\delta)$-algebra homomorphism $ (D,\partial)\lra (W(D)\otimes B,\partial^W\otimes \delta). $ \item If $B$ is a subring of $D$ and the inclusion is the structure morphism of $D$ as a $B$-algebra, then the assignment $\partial\mapsto \partial^W$ is a Lie-ring and an $A$-module homomorphism. \end{enumerate}\ \par \smallskip\noindent We expose several applications of the differential Weil descent. A major one addresses a method to produce differential fields (in commuting derivations and of characteristic 0), which possess a minimal differential closure (or in Kolchin's terminology, \textit{constraint closure}). First examples of such fields were given by Singer in \cite{Singer1978b}. He showed that for every closed ordered differential field $K$ in one derivation (cf. \cite{Singer1978a}), the algebraic closure $K[i]$ is differentially closed and minimal over $K$. The only other known examples are fixed fields of models of $\DCF_{0,m}\operatorname{A}$ the theory differentially closed fields with a generic differential automorphism, see \cite{LeoSan2016}. \ \par A key notion in this task is \textit{differentially large  field}, introduced here in \S\ref{sectionUCalgebraically}. Recall from \cite{Pop1996} that a field is \textit{large} if it is existentially closed in its Laurent series field $K((t))$; equivalently, if every smooth curve defined over $K$ with a $K$-rational point, has infinitely many $K$-rational points. Basic examples of large fields are local fields and fraction fields of local henselian rings, like the quotient field of the power series ring $K[[T_1,\ldots,T_n]]$ for any field $K$. Large fields make a remarkable appearance in the work of Pop \cite{Pop1996}, where he studies finite split embedding problems for the absolute Galois groups of a large field; for instance, he shows that if $K$ is large then every finite group is regularly realizable as a Galois group over $K(t)$. \ \par Our cue to define  differentially large fields comes from the Uniform Companion theory $\uc_m$ of differential fields of characteristic zero with $m$ commuting derivations, introduced in \cite{Tressl2005}. We fix a set $\Delta=\{\delta_1,\dots,\delta_m\}$ of commuting derivations and assume that our fields have characteristic zero. We say that a differential field $K=(K,\Delta)$ is differentially large if $K$ is large as a field and is a model of the theory $\uc_m$. The theory $\uc_m$ has rather elaborate axioms which involve somewhat deep differential-algebraic terminology, such as autoreduced and characteristic sets (see Section~\ref{prelimconven}). However, in Theorems \ref{charac} and \ref{UCalgebraically} below, we give several more practical and rather accessible characterisations of differential largeness. For instance, we prove that a large and differential field $K$ is differentially large if and only if it is existentially closed in every differential field extension $L$, provided $K$ is existentially closed in $L$ as a field. Concrete examples of differentially large fields for which our results are novel contain differentially large fields expanding real closed fields (denoted by $\CODF_m$), p-adically closed fields or pseudo-finite fields of characteristic 0, see \cite[\S8]{Tressl2005} for details. A more general mechanism on how to produce differentially large field is given in \ref{LargeElemExtension}. \ \par \ \par We give another more geometric characterisation of differential largeness in terms of smooth points of certain algebraic varieties associated to differential varieties -- the so-called ``jets''--, see Theorem \ref{charac} part (iii). In \S \ref{axiomsuc} we give a further characterisation, an algebraic-geometric one in the spirit of the Pierce-Pillay axioms for ordinary differentially closed fields \cite{PiePil1998}, see \ref{DiffLargeGeomAx}. \ \par In analogy to the algebraic case, where it is known that algebraic extensions of large fields are again large (see \cite[Proposition 1.2]{Pop1996}), we will be using the differential Weil descent and results from \S \ref{sectionUCalgebraically} to prove: \ \par \smallskip\noindent \textbf{Theorem B.} (cf. \ref{algextdifflarge}) Every algebraic extension of a differentially large field is again differentially large (with the uniquely induced derivations). \ \par \smallskip  It follows quickly that the algebraic closure of a differentially large field is in fact differentially closed.  This is used in a forthcoming paper by Aschenbrenner, Chernikov et al. to show that the theory of  differentially closed fields has a distal expansion.  As a further consequence, differentially large fields have \textit{minimal} differential closures. Finally, Theorem B together with results in \S\ref{sectionUCalgebraically}will be used in Theorem~\ref{ongroups} to show that the $K$-rational points of a connected differential algebraic group $G$ defined over a differentially large field $K$ are Kolchin-dense in $G$. \ \par \smallskip\noindent \textbf{Outlook.}We expect many more applications of the differential Weil descent in connection with differentially large fields. For instance, we expect that differentially large fields will make a notable appearance in differential Galois cohomology and the parameterised Picard-Vessiot theory for linear differential equations. This will potentially be in the form of finiteness results for cohomology groups of linear differential algebraic groups over differentially large fields with some additional properties (e.g., a differential analogue of Serre's property (F), see \cite[Ch. III, \S4]{Serre1994} used in the classical finiteness theorems for linear algebraic groups). \section{Classical Weil Descent for Algebras}\labelx{classicalweil}\labelx{ClassicalWeil}\noindent In this section we review the classical construction of Weil descent of scalars for algebras, see for example \cite[\S7.6]{BoLuRa1990}, \cite[\S2]{MooSca2010} and \cite{Grothe1995a}. For our purposes we need certain explicit formulas, so we give details. \ \par \smallskip\noindent \textbf{Convention.}Throughout, we assume our rings and algebras to be commutative and unital; ring and algebra homomorphisms are meant to be unital as well. \ \par \smallskip Let $A$ be a ring and let $B$ be an $A$-algebra. For each $A$-algebra $C$, the scalar extension by $B$ is the $B$-algebra $C\otimes_AB$ with structure map $b\mapsto 1\otimes b$ \footnote{As a general reference for tensor products, specifically in the category of algebras we refer to \cite[Appendix A]{Matsum1989}}. This assignment has a natural extension to a covariant functor $F:A\text{-}\kat{Alg}\lra B\text{-}\kat{Alg}$. The functor $F$ has a right adjoint $B\text{-}\kat{Alg}\lra A\text{-}\kat{Alg}$ given by restricting scalars. If $B$ is finitely generated and free as an $A$-module, then $F$ also has a left adjoint $W$, called \notion{Weil descent}, or \notion{Weil restriction}. We start with a reminder on left adjoints in general, ready made for use later on. \begin{FACT}{:Fact.}{LeftAdjoint} \textup{\cite[Thm 2, p.83, Cor. 1,2, p.84]{MacLan1998}}\noindent Let $F:\CC\lra \CD$ be a covariant functor between categories $\CC$ and $\CD$. \begin{enumerate}[(i)]\item The following are equivalent. \begin{enumerate}\item\labelx{LeftAdjointLA}$F$ has a left adjoint $W$, i.e., $W:\CD\lra \CC$ is a covariant functor such that for all $D\in \CD$ the functor $\Hom_\CD(D,F(\ \underline {\ }\ )):\CC\lra \kat{Sets}$ is represented by $W(D)$ meaning that the functors $\Hom_{{\CC}}({W}(D),\ \underline {\ }\ )$ and $\Hom_\CD(D,F(\ \underline {\ }\ ))$ are isomorphic \footnote{Recall that two functors are isomorphic if there is an invertible natural transformation between them.}. \item\labelx{LeftAdjointDagger}For each $D\in \CD$ there are $W(D)\in \CC$ and a $\CD$-morphism $W_D:D\lra F(W(D))$ such that the following condition holds: \ \par \smallskip \begin{enumerate}\item[$(\dagger)$:]For every $C\in\CC$ and each morphism $f:D\lra F(C)$, there is a unique $\CC$-morphism $g:{W}(D)\lra C$ such that the following diagram commutes \begin{center}\begin{tikzcd}  F(W(D)) \ar[r, dashed, "F(g)"] & F(C)\\   D \ar[u, "W_D"] \ar[ru,"f"']\end{tikzcd}\end{center}\noindent In other words, $W_D$ gives rise to a bijection \[\tau(D,C):\Hom_{{\CC}}({W}(D),C)\lra \Hom_{{\CD}}(D,{F}(C)),\ g\longmapsto {F}(g)\circ {W}_D. \]\end{enumerate}\end{enumerate}\end{enumerate}\begin{enumerate}[(i),resume]\item\labelx{LeftAdjointLAimpliesDagger}If \ref{LeftAdjointLA}  holds, then for every such functor $W$,  all $D\in\CD$ and each isomorphism $\tau(D,\ \underline{\ }\ ) :\Hom_{{\CC}}({W}(D),\ \underline {\ }\ )\lra \Hom_{\CD}(D,F(\ \underline {\ }\ ))$ as in \ref{LeftAdjointLA}, the choice $W(D)$ and $W_D=\tau(D,W(D))(\id_{W(D)})$ satisfy property $(\dagger)$ of \ref{LeftAdjointDagger}. \ \par The assignment $D\mapsto W_D$ is a natural transformation $\id_{\CD}\lra F\circ W$ and is called the \notion{unit of the adjunction}; $W_D$ is called the \notion{component} at $D$ of that unit. \ \par Similarly, for each $C\in \CC$ the morphism $F_C:W(F(C))\lra C$ that is sent to $\id_{F(C)}$ via $\tau(F(C),C)$ gives rise to a natural transformation $W\circ F\lra \id_{\CC}$, called the \notion{counit of the adjunction}; $F_C$ is called the \notion{component} at $C$ of that counit. \item\labelx{LeftAdjointDaggerImpliesLA}If \ref{LeftAdjointDagger} holds, then for any choice of objects $W(D)$ and morphisms $W_D$ as in \ref{LeftAdjointDagger}, for $D\in\CD$, the assignment $D\mapsto W(D)$ can be extended to a functor $W:\CD\lra\CC$ satisfying \ref{LeftAdjointLA} as follows: Take a morphism $f_0:D\lra D'$ and set $f:=W_{D'}\circ f_0$ and $C:=W(D')$. Then define $W(f_0):W(D)\lra W(D')$ as the unique $\CC$-morphism $W(D)\lra W(D')$ such that the diagram \begin{center}\begin{tikzcd}[row sep=7ex,column sep=10ex]  F(W(D)) \ar[r, dashed, "F(W(f_0))"] & F(C)=F(W(D'))\\   D \ar[u, "W_D"]\ar[r,"f_0"'] \ar[ru,"f"'] & D'\ar[u, "W_{D'}"]\end{tikzcd}\end{center}commutes, according to $(\dagger)$. \item\labelx{LeftAdjointUnique}Any two functors that are left adjoint to $F$ are isomorphic. \item\labelx{LeftAdjointExact}If $W$ is left adjoint to $F$, then $W$ preserves all co-limits, cf. \cite[p. 119, last paragraph]{MacLan1998}. For example $W$ preserves direct limits and fibre sums (aka pushouts). \ \par \end{enumerate}\end{FACT}\begin{FACT}{:Notation and setup.}{NotationSetup}We return to our setup of a ring $A$ and an $A$-algebra $B$. Let \[F:A\text{-}\kat{Alg}\lra B\text{-}\kat{Alg}\]be the functor defined by $F(C)=C\otimes B$ and for $\phi:C\lra C'$, $F(\phi)=\phi\otimes \id_B$. Here and below, tensor products are taken over $A$, unless stated otherwise. \ \par \smallskip\noindent We will from now on assume that $B$ is free and finitely generated as an $A$-module of dimension $\ell$ over $A$. We also fix generators $b_1,\ldots,b_\ell $ of the $A$-module $B$. For $i\in \{ 1,...,\ell \} $ let \[\lambda_i:B\lra A,\ \lambda_i (\sum_{j=1}^\ell  a_jb_j)=a_i \]be the $A$-module homomorphism dual to $A\lra B,\ a\mapsto a\mal b_i$. If $C$ is an $A$-algebra we write $\lambda_i^C=\id_C\otimes \lambda_i:C\otimes B\lra C\otimes A=C$ for the base change of $\lambda_i$ to $C$. Since $b=\sum_i \lambda_i(b)b_i$ for $b\in B$ we obtain \begin{align*}    c\otimes b=c\otimes \sum_i \lambda_i(b)b_i     =\sum_i (\lambda_i(b)c)\otimes b_i, \end{align*}for $c\in C$. Hence $\lambda_i^C(c\otimes b)=\lambda_i(b)\mal c$ is the coefficient of $c\otimes b$ at $1\otimes b_i$ when it is written in the basis $1\otimes b_1,\ldots,1\otimes b_\ell $ of the free $C$-module $C\otimes B$. This extends to all $f\in C\otimes B$, thus \[f=\sum_{i=1}^\ell \bigl(\lambda_i^C(f)\otimes b_i\bigr). \leqno{(*)}\]\ \par \end{FACT}\begin{FACT}{Definition}{DefnWBT}Let $T$ be a set of indeterminates for $A$ and $B$. We define an $A$-algebra $W(B[T])=A[T]^{\otimes \ell}\ (=\underbrace{A[T]\otimes\ldots\otimes A[T]}_{\ell\text{-times}})$. For $i\in \{1,\ldots,\ell \}$ and $t\in T$ we write \[t(i):=1\otimes ...\otimes 1\otimes \underbrace{t}_{i\text{-th position}}\otimes 1\otimes ...\otimes 1\in A[T]^{\otimes \ell}. \]Let $W_{B[T]}$ be the unique $B$-algebra homomorphism \[W_{B[T]}:B[T]\lra F(W(B[T]))=A[T]^{\otimes \ell}\otimes B\text{ with }\]\[W_{B[T]}(t)=\sum _{i=1}^\ell (t(i)\otimes b_i)\quad (t\in T). \]Further, let $F_{A[T]}$ be the unique $A$-algebra homomorphism \[F_{A[T]}:W(F(A[T]))=A[T]^{\otimes \ell}\lra A[T]\]with the property $F_{A[T]}(t(i))=\lambda_i(1)\mal t$ for $t\in T,\ i\in\{1,\ldots,\ell \}$. \end{FACT}\begin{FACT}{:Explicit description of the Weil descent of polynomial algebras.}{WeilBT}The $A$-algebra $W(B[T])$ and the morphism $W_{B[T]}$ described above satisfy condition $(\dagger)$ of \ref{LeftAdjoint}\ref{LeftAdjointDagger}. Hence by \ref{LeftAdjoint}\ref{LeftAdjointDaggerImpliesLA} we may choose $W(B[T])$ as the Weil descent of $B[T]$, and $W_{B[T]}$ as the unit of the adjunction at $B[T]$; these choices are then independent of the basis $b_1,\ldots,b_\ell$ up to a natural $A$-algebra isomorphism (see \ref{LeftAdjoint}\ref{LeftAdjointUnique}). \ \par Explicitly, for every $C\in A\text{-}\kat{Alg}$, the map \map{\tau=\tau(B[T],C):\ }{\Hom_{A\text{-}\kat{Alg}}(A[T]^{\otimes \ell},C)}{\Hom_{B\text{-}\kat{Alg}}(B[T],C\otimes B)}{\phi}{F(\phi)\circ W_{B[T]}=(\phi \otimes \id_B)\circ W_{B[T]}}is bijective, where $\phi \otimes \id_B=F(\phi):F(W(B[T]))=A[T]^{\otimes d}\otimes B\lra C\otimes B$ is the base change of $\phi $. For $t\in T$ we have \[\tau(\phi)(t)=\sum _{i=1}^\ell (\phi(t(i))\otimes b_i). \]\ \par \smallskip\noindent The compositional inverse of $\tau=\tau(B[T],C)$ is defined as follows. Let $\psi :B[T]\lra C\otimes B$ be a $B$-algebra homomorphism. We define an $A$-algebra homomorphism $\phi :A[T]^{\otimes d}\lra C$ by \[\phi (t(i)):=\lambda_i^C(\psi(t))\ (t\in T,\ i=1,\ldots,\ell ). \]Since $A[T]\otimes B\cong _BB[T]$, $\psi $ is uniquely determined by $\{\psi(t)\st t\in T\}$ and we see that $\phi $ is the unique preimage of $\psi $ under $\tau$. \ \par \noindent Further, one checks easily that $F_{A[T]}$ is the component of the counit of the adjunction at $A[T]$. \end{FACT}\begin{FACT}{:Explicit description of the Weil descent of $B$-algebras.}{WeilD}Now let $D$ be a $B$-algebra. Take a surjective $B$-algebra homomorphism  $\pi_D:B[T]\onto D$ for some set $T$ of indeterminates. Let $I_D$ be the ideal generated in $W(B[T])=A[T]^{\otimes \ell}$ generated by all the $\lambda_i^{W(B[T])}(W_{B[T]}(f))$, where $i\in\{1,\dots,\ell\}$ and $f\in\ker (\pi_D)$. We define \[W(D):=W(B[T])/I_D \]and write \[W(\pi_D):W(B[T])\lra W(D)\]for the residue map. Then the bijection $\tau(B{[}T{]}{,}C)$ from \ref{WeilBT} induces a bijection \[\tau(D{,}C):\Hom_{A\text{-}\kat{Alg}}(W(D),C)\lra \Hom_{B\text{-}\kat{Alg}}(D,F(C))\]such that the diagram \begin{center}\begin{tikzcd}[row sep=10ex,column sep=10ex]  \Hom_{A\text{-}\kat{Alg}}(W(D),C) \ar[d,"\underline{\ }\circ W(\pi_D)"', hook] \ar[r,"\tau(D{,}C)", dashed] & \Hom_{B\text{-}\kat{Alg}}(D,F(C))\ar[d,"\ \underline{\ }\circ \pi_D", hook] \\   \Hom_{A\text{-}\kat{Alg}}(W(B[T]),C)  \ar[r,"\tau(B{[}T{]}{,}C)"', "\simeq"] & \Hom_{B\text{-}\kat{Alg}}(B[T],F(C))\end{tikzcd}\end{center}commutes. \ \par \noindent The commutativity of the diagram above says that for $\phi\in \Hom_{A\text{-}\kat{Alg}}(W(D),C)$ we have \begin{align*}(+)\quad \tau(D,C)(\phi)\circ \pi_D&=\tau(B{[}T{]}{,}C)(\phi\circ W(\pi_D))\cr &=((\phi\circ W(\pi_D))\otimes\id_B )\circ W_{B[T]}. \end{align*}\ \par \noindent Finally, we display the map $W_D:=\tau(D,W(D))(\id_{W(D)}):D\lra F(W(D))$ explicitly and show that -- together with $W(D)$ -- it satisfies the mapping property of $(\dagger)$ in \ref{LeftAdjoint}\ref{LeftAdjointDagger}. Take $t\in T$. Then by (+) with $C=W(D),\ \phi=\id_{W(D)}$ we see that \[W_D(\pi_D(t))=\sum _{i=1}^\ell  W(\pi_D)(t(i))\otimes b_i=\sum _{i=1}^\ell  (t(i)\mod I_D)\otimes b_i.\leqno{(\sharp)}\]\ \par \noindent Pick an $A$-algebra $C$. Since $\tau(D,C)$ is bijective, the mapping property of $(\dagger)$ in \ref{LeftAdjoint}\ref{LeftAdjointDagger} follows after checking $\tau(D,C)(\phi)=F(\phi)\circ W_D$ for all $\phi\in \Hom_{A\text{-}\kat{Alg}}(W(D),C)$. Using $(\sharp)$ this is a straightforward computation. \ \par Using \ref{LeftAdjoint}\ref{LeftAdjointDaggerImpliesLA},\ref{LeftAdjointUnique}we have justified our choice of $W(D)$ and $W_D$ for the Weil descent. Finally, \ref{LeftAdjoint}\ref{LeftAdjointDaggerImpliesLA} gives the definition of $W$ on morphisms. \end{FACT}\section{Differential Weil Descent}\label{differentialweil}\noindent In this section we present a construction of a Weil descent functor in the category of differential algebras in arbitrary characteristic. We first recall some basic facts about differential algebras and their tensor products. We continue to assume that our rings and algebras are unital and commutative. \begin{FACT}{:Generalities about differential algebra.}{GenDiff}The following are well known generalities on differential algebras whose proofs are straightforward. For a ring $A$ we let $\Der(A)$ denote the family of derivations on $A$. \begin{enumerate}[(i)]\item\labelx{GenDiffExtendOrdinary}Let $A$ be a ring and let $T$ be a not necessarily finite set of indeterminates over $A$. For each $t\in T$ let $f_t\in A[T]$. Let $d\in \Der(A)$. Then there is a unique derivation $\delta $ of $A[T]$ extending $d$ with $\delta(t)=f_t$ for all $t\in T$. \ \par \ \par \end{enumerate}\ \par \noindent For $ d,\delta\in\Der(A)$ we write $[ d,\delta]:A\lra A$ for the Lie-bracket of $ d $ and $\delta$, defined by $[ d,\delta](a)= d\delta(a)-\delta d(a)$. Notice that $[ d,\delta]$ is again a derivation of $A$. \begin{enumerate}[(i),resume]\item \labelx{GenDiffLieFromGenerators}Let $A$ be a ring and let $S\subseteq A$ be a set of generators of the ring $A$. \begin{enumerate}[(a)]\item Let $ d_1, d_2,\delta_1,\ldots,\delta_n\in\Der(A)$ and suppose there are $a_{i}\in A$ with $[ d_1, d_2](s)=\sum_{i=1}^na_{i}\delta_i(s)$ for all $s\in S$. Then $[ d_1, d_2]=\sum_{i=1}^na_{i}\delta_i$. \ \par \item Let $\phi:A\lra B$ be a ring homomorphism and let $ d:A\lra A,\ \delta:B\lra B$ be derivations. If $\phi( d s)=\delta(\phi(s))$ for all $s\in S$, then $\phi$ is a differential homomorphism $(A, d)\lra (B,\delta)$. \ \par \end{enumerate}\item \labelx{GenDiffTensor}Let $d\in\Der(A)$ and let $(B,\delta), (C,\partial)$ be differential $(A,d)$-algebras. Then there is a unique derivation $\delta\otimes \partial$ on $B\otimes_AC$ such that the natural maps $B\lra B\otimes_AC, C\lra B\otimes_AC$ are differential maps, cf. \cite[Chapter 2 (1.1), p. 21]{Buium1994}. \ \par \item\labelx{GenDiffLieAndTensor}Now let $d_1,d_2\in\Der(A)$, $\delta_1,\delta_2\in\Der(B)$ and $\partial_1,\partial_2\in\Der(C)$ such that $(B,\delta_i),(C,\partial_i)$ are differential $(A,d_i)$-algebras. Then, for $a_1,a_2\in A$, straightforward checking shows that \begin{enumerate}\item $(a_1\delta_1+a_2\delta_2)\otimes (a_1\partial_1+a_2\partial_2)=a_1(\delta_1\otimes \partial_1)+ a_2(\delta_2\otimes \partial_2)$. \item $[\delta_1,\delta_2]\otimes [\partial_1,\partial_2]=[\delta_1\otimes \partial_1,\delta_2\otimes \partial_2].$ \end{enumerate}\end{enumerate}\end{FACT}\bigskip \ \par \noindent As in Section \ref{classicalweil} we work with a ring $A$ and an $A$-algebra $B$ that is free and finitely generated by $b_1,\ldots,b_\ell $ as an $A$-module. We fix a derivation $d$ on $A$ and a derivation $\delta$ on $B$ such that $(B,\delta)$ is a differential $(A,d)$-algebra (meaning that the structure map $A\lra B$ is differential). \ \par By \ref{GenDiff}\ref{GenDiffTensor}, for any differential $(A,d)$-algebra $(C,\partial_C)$, there is a unique derivation $\partial_C\otimes \delta$ on $F(C)=C\otimes B$ such that the natural map $C\lra F(C)$ is a differential $(B,\delta)$-algebra morphism. \begin{FACT}{Theorem}{partialDW}Let $(D,\partial_D)$ be a differential $(B,\delta)$-algebra. Then there is a unique derivation $\partial_D^W$ on $W(D)$ such that $(W(D),\partial_D^W)$ is a differential $(A,d)$-algebra and \[W_D:(D,\partial_D)\lra (F(W(D)),\partial_D^W\otimes \delta)\]is a differential $(B,\delta)$-algebra homomorphism, i.e., $W_D\circ \partial_D=(\partial_D^W\otimes \delta)\circ W_D$. \ \par Furthermore, $\partial_D^W$ only depends on $\partial_D$ and not on $\delta$. \end{FACT}\begin{proof}Take any set $T$ of differential indeterminates and a surjective $(B,\delta)$-algebra homomorphism $\pi_D:(B\{T\},\partial)\onto (D,\partial_D)$. Here, the differential polynomial ring $B\{T\}$ is considered just as polynomial ring over $B$ in the algebraic indeterminates $t_\theta$, where $t\in T$ and $\theta\in\Theta:=\{\partial^i:i\geq 0\}$. Further, $\partial=\partial_{B\{T\}}:B\{T\}\lra B\{T\}$ is the natural derivation, thus $\partial t_\theta =t_{\partial\theta}$. \ \par We choose $W_{B\{T\}}:B\{T\}\lra F(W(B\{T\}))$ according to \ref{DefnWBT} for the set of indeterminates $\{t_\theta\st t\in T,\ \theta\in\Theta\}$ and $W_D:D\lra F(W(D))$ according to \ref{WeilD}. Also recall $(\sharp)$ in \ref{WeilD}, which says that  $W_D(\pi_D(t_\theta))=\sum _{i=1}^\ell  W(\pi_D)(t_\theta(i))\otimes b_i$ \ \par  \smallskip \ \par \noindent \Claim 1 If $\eps:W(D)\lra W(D)$ is a derivation such that $(W(D),\eps)$ is a differential $A$-algebra, then for all $t\in T$ and any $\theta\in \Theta$ we have \[((\eps\otimes \delta)\circ W_D)(\pi_D(t_\theta))= \sum_{i=1}^\ell \biggl(\eps (W(\pi_D)(t_\theta(i)))+\sum_{j=1}^\ell \lambda_i(\delta b_j)\mal W(\pi_D)(t_\theta(j))  \biggr)\otimes b_i. \]See \ref{GenDiff}\ref{GenDiffTensor} for the definition of $\eps\otimes\delta$. \ \par \noindent \textit{Proof.}This is a straightforward calculation using $\delta b_i=\sum_{j=1}^\ell \lambda_j(\delta b_i)b_j$. \hfill$\diamond$ \ \par \smallskip\noindent \Claim 2 If $\eps:W(D)\lra W(D)$ is a derivation such that $(W(D),\eps)$ is a differential $A$-algebra, then $W_D\circ \partial=(\eps\otimes \delta)\circ W_D$ if and only if for all $t_\theta(i)$ we have \[\eps(W(\pi_D)(t_\theta(i)))= W(\pi_D)(t_{\partial\theta}(i)-\sum_{j=1}^\ell \lambda_i(\delta(b_j))\mal W(\pi_D)(t_\theta(j)).\leqno{(*)}\]\noindent \textit{Proof.}By \ref{GenDiff}\ref{GenDiffLieFromGenerators}(b), $W_D\circ \partial=(\eps\otimes \delta)\circ W_D$ if and only if $((\eps\otimes \delta)\circ W_D)(\pi_D(t_\theta))=(W_D\circ \partial)(\pi_D(t_\theta))$ for all $t_\theta$. By Claim 1 this is equivalent to \begin{align*}\sum_{i=1}^\ell &\biggl(\eps (W(\pi_D)(t_\theta(i)))+\sum_{j=1}^\ell \lambda_i(\delta b_j)\mal W(\pi_D)(t_\theta(j))  \biggr)\otimes b_i\cr &=W_D(\partial(\pi_D(t_\theta)))\cr &=W_D(\pi_D(\partial t_\theta)),\text{ since }\pi_D\text{ is a differential map}\cr &=W_D(\pi_D(t_{\partial\theta}))\cr &=\sum _{i=1}^\ell  W(\pi_D)(t_{\partial\theta}(i))\otimes b_i\text{, by }(\sharp)\text{ in \ref{WeilD}}. \end{align*}Since $1\otimes b_1,\ldots,1\otimes b_\ell$ is a basis of $F(W(D))$ over $W(D)$, the identity is equivalent to $(*)$ being true for all $i\in\{1,\ldots,\ell\}$. \hfill$\diamond$ \ \par \medskip\noindent Claim 2 implies the uniqueness statement of the \factname, because the set of all the $W(\pi_D)(t_\theta(i))$ generates $W(D)$. For existence, we first deal with $B\{T\}$ instead of $D$. In that case, Claim 2 says that we only need to find a derivation $\partial_{B\{T\}}^W$ on $W(B\{T\})$ such that $(W(B\{T\}),\partial_{B\{T\}}^W)$ is a differential $(A,d)$-algebra with the property \[\partial_{B\{T\}}^W(t_\theta(i))=t_{\partial \theta}(i)-\sum_{j=1}^\ell \lambda_i(\delta(b_j))\mal t_\theta(j). \]By \ref{GenDiff}\ref{GenDiffExtendOrdinary} applied to the polynomial ring $W(B\{T\})$ over $A$, such a derivation indeed exists.\footnote{Notice that $W(B\{T\})$ naturally is a differential polynomial ring over $A$, but $\partial_{B\{T\}}^W$ is in general not the natural derivation of $W(B\{T\})$.}\ \par \medskip\noindent It remains to prove that there is a derivation $\partial_D^W$ of $W(D)$ as required. \ \par \smallskip\noindent \Claim 3 The ideal $I_D$ of $W(B\{T\})$ (see \ref{WeilD}) is a differential ideal for $\partial_{B\{T\}}^W$. \ \par \noindent \textit{Proof.}Let $f\in \ker(\pi_D)$. Then $W_{B\{T\}}(f)=\sum_{i=1}^\ell g_i\otimes b_i$, where $g_i=\lambda_i^{W(B\{T\})}(W_{B\{T\}}(f))$. By definition of $I_D$ it suffices to show that $\partial_{B\{T\}}^W(g_i)\in I_D$ \footnote{Notice that the module homomorphism $\lambda_i^{W(B(\{T\}))}:F(W(B\{T\}))\lra W(B\{T\})$ does not in general commute with the derivations.}. Now one checks that \[W_{B\{T\}}(\partial_{B\{T\}}(f))=\sum_{i=1}^\ell\biggl(\partial_{B\{T\}}^W(g_i)+\sum_{j=1}^\ell \lambda_i(b_j)g_j\biggr)\otimes b_i \]Since $1\otimes b_1,\ldots,1\otimes b_\ell$ is a basis of $F(W(B\{T\}))$ over $W(B\{T\})$ we see that \[\lambda_i^{W(B\{T\})}(W_{B\{T\}}(\partial_{B\{T\}}(f)))=\partial_{B\{T\}}^W(g_i)+\sum_{j=1}^\ell \lambda_i(b_j)g_j. \]The left hand side here is in $I_D$ by definition of $I_D$ and because $\ker(\pi)$ is differential for $\partial_{B\{T\}}$. As all $g_i\in I_D$ this entails $\partial_{B\{T\}}^W(g_i)\in I_D$. \hfill$\diamond$ \ \par \smallskip\noindent By Claim 3, the derivation $\partial_{B\{T\}}$ induces a derivation $\delta_D^W$ of $W(D)=W(B\{T\})/I_D$ such that $(W(D),\partial_D^W)$ is a differential $(A,d)$-algebra. It remains to show that $W_D$ is a differential $(B,\delta)$-algebra homomorphism, i.e., $W_D\circ \partial_D=(\partial_D^W\otimes \delta)\circ W_D$. This can be seen by a diagram chase as follows. Consider the diagram of maps \begin{center}\begin{tikzcd}[row sep=30pt, column sep=50pt]\ & D \arrow[dl, leftarrow, "\pi"'] \arrow[rr, "W_D"] \arrow[dd, "\partial_D" near end ] & \ & W(D)\otimes B \arrow[dl, leftarrow, "W(\pi)\otimes \id_B"'] \arrow[dd, "\partial^W_D\otimes \delta"]\\   {B\{T\}} \arrow[rr, crossing over, "W_{B\{T\}}" near end] \arrow[dd,"\partial_{B\{T\}}"'] & \ & W({B\{T\}})\otimes B \\ \ & D \arrow[dl, leftarrow, "\pi"'] \arrow[rr, "W_D"' near start] & \ & W(D)\otimes B \arrow[dl, leftarrow, "W(\pi)\otimes \id_B"] \\   {B\{T\}} \arrow[rr, "W_{B\{T\}}"'] &\ & W({B\{T\}})\otimes B \arrow[from=uu, crossing over, "\partial^W_{B\{T\}}\otimes \delta" near start]\\ \end{tikzcd}\end{center}The claim is that the back side of this cube is commutative. Now, all other sides of the cube are commutative squares, because \begin{itemize}\item Bottom and top of the cube are identical and commute as a property of the classical Weil descent. \item The front of the cube commutes as we know the \factname\ already for $(B\{T\},\partial_{B\{T\}})$. \item The square on the left hand side commutes by choice of $(B\{T\},\partial_{B\{T\}})$. \item The square on the right hand side commutes by applying base change to $B$ to the definition of $\partial_D^W$. \end{itemize}Since $\pi$ is surjective, we see that the back of the cube also commutes. This finishes the proof of existence and uniqueness of $\partial_W$. From Claim 2 we see that the definition of $\partial_{B\{T\}}^W$ only depends on $\partial_{B\{T\}}$ and not on $\delta$, because the structure map $B\lra D$ is differential. But then by construction of $\partial_{D}^W$ after Claim 3, $\partial_{D}^W$ only depends on $\partial_D$ and not on $\delta$. \end{proof}\begin{FACT}{Theorem}{liebracket}Let again $B$ be an $A$-algebra that is finitely generated and free as an $A$-module and let $D$ be a $B$-algebra. \ \par Let $\Der_B(D)$ be the set of all $\partial\in \Der(D)$ for which there are derivations $d$ of $A$ and $\delta$ of $B$ such that the structure maps of $B$ and $D$ are differential maps $(A,d)\lra (B,\delta)$ and $(B,\delta)\lra (D,\partial)$, respectively.\footnote{If the structure morphism of $D$ as a $B$-algebra is injective and we think of the structure maps $A\lra B$ and $B\lra D$ as inclusions, then $\Der_B(D)$ is the set of all derivations of $D$ that restrict to derivations on $A$ and $B$.}\ \par Then $\Der_B(D)$ is an $A$-submodule  and a Lie subring of $\Der(D)$ and the map $\Der_B(D)\lra \Der(W(D))$ that sends $\partial$ to the derivation $\partial^W$ defined in \ref{partialDW}, is an $A$-module and a Lie ring homomorphism. \ifoldversion\OldStart Let $M=\Der(A)\times \Der(B)\times \Der(D)$. Then $M$ is an $A$-module and a Lie-ring with the coordinate-wise Lie bracket. Let $\Der(A,B,D)$ be the set of all $(d,\delta,\partial)\in M$ such that the structure maps of $B$, $D$ are differential maps $(A,d)\lra (B,\delta)$ and $(B,\delta)\lra (D,\partial)$, respectively. Then $\Der(A,B,D)$ is an $A$-submodule  and a Lie subring of $M$ and the map $\Der(A,B,D)\lra \Der(W(D))$ that sends $(d,\delta,\partial)$ to the derivation $\partial^W$ defined in \ref{partialDW}, is an $A$-module and a Lie ring homomorphism. \ \par Explicitly, given $(d_1,\delta_1,\partial_1),(d_2,\delta_2,\partial_2)\in \Der(A,B,D)$ we have \OldEnd\else\fi Explicitly, given $\partial_1,\partial_2\in \Der_B(D)$ we have \begin{enumerate}[(i)]\item $(a_1\partial_1+a_2\partial_2)^W=a_1\partial_1^W+a_2\partial_2^W$ for all $a_1,a_2\in A$. \item $[\partial_1,\partial_2]^W=[\partial_1^W,\partial_2^W]$. In particular, $\partial_1^W,\partial_2^W$ commute if $\partial_1,\partial_2$ commute. \end{enumerate}\end{FACT}\begin{proof}In each case, the derivation of $W(D)$ on the right hand side turns $W(D)$ into a differential $A$-algebra, when $A$ is furnished with the derivation $a_1d_1+a_2d_2$ and $[d_1,d_2]$ respectively. By uniqueness in \ref{partialDW} we thus only need to verify the defining equation of the left hand side for the right hand side. \ \par \smallskip\noindent (i) Using \ref{GenDiff}\ref{GenDiffLieAndTensor}(a)we get $((a_1\partial_1^W+a_2\partial_2^W)\otimes (a_1\delta_1+a_2\delta_2))\circ W_D= a_1(\partial_1^W\otimes \delta_1)\circ W_D+a_2(\partial_2^W\otimes \delta_2)\circ W_D= a_1W_D\circ \partial_1+a_2W_D\circ \partial_2=W_D\circ (a_1\partial_1+a_2\partial_2)$, since $W_D$ is an $A$-algebra homomorphism. \ \par \smallskip\noindent (ii) Using \ref{GenDiff}\ref{GenDiffLieAndTensor}(b) we get $([\partial_1^W,\partial_2^W]\otimes [\delta_1,\delta_2])\circ W_D= [\partial_1^W\otimes \delta_1,\partial_2^W\otimes \delta_2]\circ W_D= (\partial_1^W\otimes \delta_1)\circ (\partial_2^W\otimes \delta_2)\circ W_D- (\partial_2^W\otimes \delta_2)\circ (\partial_1^W\otimes \delta_1)\circ W_D= (\partial_1^W\otimes \delta_1)\circ W_D\circ \partial_2- (\partial_2^W\otimes \delta_2)\circ W_D\circ \partial_1= W_D\circ \partial_1\circ \partial_2- W_D\circ \partial_2\circ \partial_1=W_D\circ [\partial_1,\partial_2]. $ \end{proof}\ \par \ifoldversion\OldStart \begin{FACT}{Definition}{DefnDiffWeil}\textcolor{red}{merge to 3.6 and update labels}Let $(A,d_1,\dots,d_m)$ be a differential ring \textcolor{red}{Is it commonly understood what this means?}and $(B,\delta_1,\dots,\delta_m)$ a differential $(A,d_1,\dots,d_m)$-algebra that is finite and free as an $A$-module. For $(D,\partial_1,\dots,\partial_m)$ a differential $(B,\delta_1,\dots,\delta_m)$-algebra, we define the \notion{differential Weil descent} (or \notion{Weil restriction}) of $D$ from $B$ to $A$, denoted $W^{\operatorname{diff}}(D)$, to be the algebra $W(D)$ equipped with the derivations $\partial_1^W,\dots,\partial_m^W$ from Theorem \ref{partialDW}. \end{FACT}We conclude this section with an important consequence of Theorem \ref{partialDW} and Fact \ref{liebracket} that justifies Definition \ref{DefnDiffWeil}. \begin{FACT}{Corollary}{DiffDescentMain}\textcolor{red}{merge into 3.6 and update labels}Let $(A,d_1,\dots,d_m)$, $(B,\delta_1,\dots,\delta_m)$, and $(D,\partial_1,\dots,\partial_m)$ be as in Definition~\ref{DefnDiffWeil}. If $(C,\eta_1,\dots,\eta_m)$ is a differential $(A,d_1,\dots,d_m)$-algebra, then the bijection \[\tau(D{,}C):\Hom_{A\text{-}\kat{Alg}}(W(D),C)\lra \Hom_{B\text{-}\kat{Alg}}(D,F(C)),\ \phi\longmapsto {F}(\phi)\circ {W}_D \]from \ref{WeilD} restricts to a bijection between the differential $(A,d_1,\dots,d_m)$-algebra homomorphisms from $W^{\operatorname{diff}}(D)$ to $(C,\eta_1,\dots,\eta_m)$ and the differential $(B,\delta_1,\dots,\delta_m)$-algebra homomorphisms from $(D,\partial_1,\dots,\partial_m)$ to $F^{\operatorname{diff}}(C):=(F(C),\eta_1\otimes\delta_1,\dots,\eta_m\otimes\delta_m)$. Also, if $f:(D,\partial_1,\dots,\partial_m)\to (D',\partial_1',\dots,\partial_m')$ is a differential $(B,\delta_1,\dots,\delta_m)$-algebra homomorphism, then $W(f):W^{\operatorname{diff}}(D)\to W^{\operatorname{diff}}(D')$ is a differential $(A,d_1,\dots,d_m)$-algebra homomorphism. \ \par Furthermore, if $[\partial_i,\partial_j]=a_1\partial_1+\cdots+a_m\partial_m$ for some $a_k\in A$, then $[\partial_i^W,\partial_j^W]=a_1\partial_1^W+\cdots+a_m\partial_m^W$. In particular, if the derivations on $D$ commute then so do the derivations on $W^{\operatorname{diff}}(D)$. \end{FACT}\OldEnd\else\fi \ \par \noindent Theorems \ref{partialDW} and \ref{liebracket} establish \begin{FACT}{:The differential Weil descent.}{DiffWeilDescent}Let $A$ be a ring and let $d_1,\ldots,d_m\in \Der(A)$. We write $d=(d_1,\ldots,d_m)$. A differential $(A,d)$-algebra is an $A$-algebra $C$ together with $\eta_1,\ldots,\eta_m\in \Der(C)$, such that the structure map $A\lra C$ is a differential morphism $(A,d_i)\lra (C,\eta_i)$ for all $i\in\{1,\ldots,m\}$. Let $(A,d)\text{-}\kat{Alg}$ be the category of differential $(A,d)$-algebras whose morphisms are ring homomorphisms preserving the appropriate derivations. \ \par We fix a differential $(A,d)$-algebra $(B,\delta)$ such that $B$ is finitely generated and free as an $A$-module. Then \begin{enumerate}[(i)]\item The functor $F^\mathrm{diff}:(A,d)\text{-}\kat{Alg}\lra (B,\delta)\text{-}\kat{Alg}$ that sends $(C,\eta)$ to $(C\otimes B,\eta_1\otimes \delta_1,\ldots,\eta_m\otimes \delta_m)$ has a left adjoint $W^\mathrm{diff}:(B,\delta)\text{-}\kat{Alg}\lra (A,d)\text{-}\kat{Alg}$, which we call the \notion{differential Weil descent}(or differential Weil restriction)from $(B,\delta)$ to $(A,d)$. It sends $(D,\partial)$ to $(W(D),\partial ^W)$ with $\partial ^W=(\partial_1^W,\ldots,\partial_m^W)$, $\partial_i^W$ as defined in \ref{partialDW}, and a morphism $f$ to $W(f)$. \item Let $(C,\eta)\in (A,d)\text{-}\kat{Alg}$ and let $(D,\partial)\in (B,\delta)\text{-}\kat{Alg}$. Then the bijection \[\tau(D{,}C):\Hom_{A\text{-}\kat{Alg}}(W(D),C)\lra \Hom_{B\text{-}\kat{Alg}}(D,F(C)),\ \phi\longmapsto {F}(\phi)\circ {W}_D \]from the classical Weil descent \ref{WeilD} restricts to a bijection \[\Hom_{(A,d)\text{-}\kat{Alg}}(W^\mathrm{diff}(D,\partial),(C,\eta))\lra \Hom_{(B,\delta)\text{-}\kat{Alg}}((D,\partial),F^\mathrm{diff}(C,\eta)). \]\item If $(D,\partial)\in (B,\delta)\text{-}\kat{Alg}$ and the $\partial_1,\ldots,\partial_m$ are Lie commuting with structure coefficients $a_{ij}^k\in A$ $(1\leq i,j,k\leq m)$, i.e., \[[\partial_i,\partial_j]=\sum_{k=1}^ma_{ij}^k\partial_k \quad (1\leq i,j\leq m), \]then also the derivations $\partial_1^W,\ldots,\partial_m^W$ of $W(D)$ are Lie commuting with structure coefficients $a_{ij}^k$. \end{enumerate}\end{FACT}\begin{proof}By Theorem \ref{partialDW}, the map $W_D:D\to F(W(D))$ is differential. Hence, if a morphism $\phi:W(D)\to C$ is differential, so is ${F}(\phi)\circ {W}_D$. Thus the map $\tau(D,C)$ restricts to differential morphisms as claimed in (ii). Now recall from \ref{WeilD} (and \ref{LeftAdjoint}\ref{LeftAdjointDaggerImpliesLA})that $W(f)$ is the unique map that corresponds to the morphism $W_{D'}\circ f:D\to F(W(D'))$ under the bijection $\tau(D,W(D'))$. As the latter morphism is differential, $W(f)$ must be differential. This entails (i), see \ref{LeftAdjoint}\ref{LeftAdjointLA}. Item (iii) follows immediately from \ref{liebracket}. \end{proof}\section{Differential Algebraic Setup}\label{prelimconven}\ \par \noindent One of our applications of the differential Weil descent constructed in Section~\ref{differentialweil} is to show that every algebraic extension of a large field that is a model of $\uc_m$ is again a model of $\uc_m$. This will be achieved in a similar manner to the fact that algebraic extensions of large fields are again large \cite[Proposition 2.1]{Pop1996}. The latter uses classical Weil descent and, in particular, the fact that the Weil functor preserves smoothness \cite[Appendix 2]{Oester1984}. Before presenting the applications we will review the theory $\uc_m$ and establish some useful characterizations. But first some preliminaries. \begin{FACT}{:Some preliminaries and conventions.}{}We fix a distinguished set of commuting derivations $\Delta=\{\delta_1,\dots,\delta_m\}$. We assume that all our fields are of characteristic zero. We work inside a large (saturated) differentially closed field $(\U,\Delta)$, and $K$ denotes a differential subfield of $\U$. A \notion{Kolchin-closed} subset of $\U^n$ is the common zero set of a set of differential polynomials over $\U$ in $n$ differential variables; such sets are also called \notion{affine differential varieties}. If the definining polynomials can be chosen with coefficients in $K$ we will say the set is \notion[]{defined over $K$}. The Kolchin-closed sets (defined over $K$) are the closed sets of a topology, called the \notion{Kolchin-topology} of $\U^n$ (over $K$). \ \par By a \notion{differential variety} $V$ we mean a topological space which has as finite open cover $V_1,\dots,V_s$ with each $V_i$ homeomorphic to an affine differential variety (inside some power of $\U$)such that the transition maps are regular as differential morphisms; see \cite[Chap. 1, section 7]{LeoSan2013}. We will say that the differential variety is over $K$ when all objects and morphisms can be defined over $K$. This definition also applies to our use of algebraic varieties, replacing Kolchin-closed with Zariski-closed in powers of $\U$ (recall that $\U$ is algebraically closed and a universal domain for algebraic geometry in Weil's ``foundations'' sense). \end{FACT}\begin{FACT}{Remark}{}Suppose $L/K$ is a finite field extension. Recall that the derivations $\delta_1,\dots,\delta_m$ extend uniquely from $K$ to $L$. \begin{enumerate}\item[(i)] Given a differential $L$-algebra $D$, by \ref{DiffWeilDescent}, there is a natural one-to-one correspondence between the differential $L$-points of $D$ and the differential $K$-points of $W^{\operatorname{diff}}(D)$. \item[(ii)] In the case when $D$ is the differential coordinate ring of an affine differential variety, say $D=L\{V\}$, and when $L$ has a $K$-basis $b_1,\dots,b_\ell$ of constants (meaning that $\delta_i(b_j)=0$ for all $i,j$),     then a construction     of the differential Weil descent $W^{\operatorname{diff}}(L\{V\})$ appears in \cite[\S5]{LeSMos2016}.     However, a basis of constants does not always exist, see Example \ref{example1} below. \end{enumerate}\end{FACT}\begin{FACT}{Example}{example1}We work in the ordinary case $\Delta=\{\delta\}$. Let $K=\mathbb Q(t)$ with $\delta t=1$ and consider the finite extension $L=K(b)$ where $b^2=t$. Then the (unique) induced derivation on $L$ is given by $\delta b=\frac{1}{2b}=\frac{b}{2t}$. Fix the basis $\{1,b\}$ of $L$ as a $K$-module. Consider the differential variety $V$ given by $\delta x=0$ (i.e., $V$ is simply the constants of $\U$) viewed as a differential variety over $L$. The differential Weil descent $W^{\operatorname{diff}}(V)$ is obtained as follows; write $x$ as $x_1+x_2b$ and compute $$\delta(x_1+bx_2)=\delta x_1+(\delta b) x_2+b \delta x_2=\delta x_1+\frac{b}{2t} x_2+b \delta x_2=\delta x_1+\left(\frac{x_2}{2t}+\delta x_2\right)b.$$ Thus, $W^{\operatorname{diff}}(V)$ is the affine differential variety over $K$ given by the equations $$\delta x_1=0\quad \text{ and }\quad \delta x_2+\frac{x_2}{2t}=0.$$ Note that this is not contained in a product of the constants, as one might expect. Of course, if $\delta(b)$ were zero we would instead obtain the equations $\delta x_1=0$ and $\delta x_2=0$ (which would occur if $\delta$ were trivial on $K$, for instance). \end{FACT}\begin{FACT}{void}{}We fix integers $n>0$ and $r\geq 0$, and set $$ \Gamma_n(r) = \{(\xi,i) \in \mathbb N^m\times\{1,\dots,n\} \st \sum_{i=1}^m \xi_i \leq r\}. $$ \ \par \noindent We make use of prolongation spaces and recall the definition and some properties. The \notion[]{$r$-th nabla map} $\nabla_r:\U^n\to \U^{\alpha(n,r)}$ with $\alpha(n,r):=|\Gamma_n(r)|=n\cdot\binom{r+m}{m}$ is defined by $$\nabla_r(x)= (\delta^\xi x_i:\,(\xi,i)\in \Gamma_n(r)),$$ where $x=(x_1,\dots,x_n)$ and $\delta^\xi=\delta_1^{\xi_1}\cdots\delta_m^{\xi_m}$. We order the elements of the tuple $(\delta^\xi x_i:\,(\xi,i)\in \Gamma_n(r))$ according to the canonical orderly ranking of the indeterminates $\delta^\xi x_i$; that is, \begin{equation}\label{ordercanonical}\delta^{\xi}x_i< \delta^{\zeta}x_j \iff \left(\sum \xi_k,i,\xi_1,\dots,\xi_m\right)<\left(\sum \zeta_k,j,\zeta_1,\dots,\zeta_m\right)\end{equation}where the ordering on the right-hand-side is the lexicographic one. \ \par Let $\U_r:=\U[\epsilon_1,\dots,\epsilon_m]/(\epsilon_1,\dots,\epsilon_m)^{r+1}$ where the $\epsilon_i$'s are indeterminates, and let $e:\U\to \U_r$ denote the ring homomorphism $$x\mapsto \sum_{\xi\in\Gamma_1(r)}\frac{1}{\xi_1!\cdots\xi_m!}\; \delta^\xi(x)\; \epsilon_1^{\xi_1}\cdots\epsilon_m^{\xi_m}.$$ We call $e$ the exponential $\U$-algebra structure of $\U_r$. To distinguish between the standard and the exponential algebra structure on $\U_r$, we denote the latter by $\U_r^e$. \end{FACT}\begin{FACT}{Definition}{}Given an algebraic variety $X$ the $r$-th \notion{prolongation} $\tau X$ is the algebraic variety given by the taking the $\U$-rational points of the Weil descent of $X\times_\U \U_r^e$ from $\U_r$ to $\U$. Note that the base change $V\times_\U \U_r^e$ is with respect to the exponential structure while the Weil descent is with respect to the standard $\U$-algebra structure. \end{FACT}For details and properties of prolongation spaces we refer to \cite[\S 2]{MoPiSc2008}; for a more general presentation see \cite{MooSca2010}. In particular, it is pointed out there that the prolongation $\tau_r X$ always exist when $X$ is quasi-projective (an assumption that we will adhere to later on). A characterising feature of the prolongation is that for each point $a\in X=X(\U)$ we have $\nabla_r(a)\in \tau_r X$. Thus, the map $\nabla_{r}:X\to \tau_r X$ is a differential regular section of $\pi_r:\tau_r X\to X$ the canonical projection induced from the residue map $\U_r\to \U$. We note that if $X$ is defined over the differential field $K$ then $\tau_r X$ is defined over $K$ as well. \ \par In fact, $\tau_r$ as defined above is a functor from the category of algebraic varieties over $K$ to itself, and the maps $\pi_r:\tau_r X\to X$ and $\nabla_r:X\to \tau_r X$ are natural. The latter means that for any morphism of algebraic varieties $f:X\to Y$ we get \begin{equation}\label{natural}f\circ\pi_{r,X}=\pi_{r,Y}\circ\tau_r f \quad \text{ and }\quad \tau_r f\circ\nabla_{r,X}=\nabla_{r,Y}\circ f. \end{equation}If $G$ is an algebraic group, then $\tau_r G$ also has the structure of an algebraic group. Indeed, since $\tau_r$ commutes with products, the group structure is given by $$\tau_r(*):\tau_r G\times\tau_r G\to \tau_r G$$ where $*$ denotes multiplication in $G$. Moreover, by the right-most equality in \eqref{natural}, the map $\nabla_r:G\to\tau_r G$ is an injective group homomorphism. Hence, $\nabla_r(G)$ is a differential algebraic subgroup of $\tau_r G$. \ \par Assume that $V$ is a differential variety which is given as a differential subvariety of an algebraic variety $X$. We define the $r$-th jet of $V$ to be the Zariski-closure of the image of $V$ under the $r$-th nabla map $\nabla_r:X\to \tau_r X$; that is, $$\jet_r V=\overline{\nabla_r(V)}^{\operatorname{Zar}}\subseteq \tau_r X.$$ The jet sequence of $V$ is defined as $(\jet_r V:r\geq 0)$. Note that this sequence determines $V$, indeed $$V=\{a\in X: \nabla_r(a)\in \jet_r V \text{ for all $r\geq 0$}\}.$$ In the case when $V$ is affine, say a Kolchin-closed subset of $\U^n$, and defined by differential polynomials of order at most $r$, then $$V=\{a\in \U^n: \nabla_r(a)\in \jet_r V\}.$$ \begin{FACT}{:Assumption.}{}From now on we assume, whenever necessary for the existence of jets, that our differential varieties are given as differential subvarieties of quasi-projective algebraic varieties. Of course, in the affine case this is always the case. It is worth noting that for connected differential algebraic groups this is also true. Indeed, by \cite[Corollary 4.2(ii)]{Pillay1997} every such group embeds into a connected algebraic group and the latter is quasi-projective by Chevalley's theorem. \end{FACT}\medskip \begin{FACT}{:Reminder on characteristic sets.}{}We recall some of the theory of characteristic set of prime differential ideals of the differential polynomial ring $K\{x\}$ with $x=(x_1,\dots,x_n)$. For a detailed reference we refer the reader to \cite[Chapters I and IV]{Kolchi1973}. Let $f\in K\{x\}$ be nonconstant. The leader of $f$, denoted $v_f$, is the highest ranking algebraic indeterminate that appears in $f$ (according to the canonical orderly ranking of the indeterminates $\delta^\xi x_i$, as in the equivalence \eqref{ordercanonical}). The leading degree of $f$, denoted $d_f$, is the degree of $v_f$ in $f$. The rank of $f$, denoted $\operatorname{rk}(f)$, is the pair $(v_f,d_f)$. The set of ranks is ordered lexicographically. The separant of $f$, denoted $S_f$, is the formal partial derivative of $f$ with respect to $v_f$. The initial of $f$, denoted $I_f$, is the leading coefficient of $f$ when viewed as a polynomial in $v_f$. Note that both $S_f$ and $I_f$ have lower rank than $f$. Given a finite subset $\Lambda\subseteq K\{x\}\setminus K$, we set $H_\Lambda:=\prod_{f\in \Lambda}I_fS_f$. \ \par One says that $g\in K\{x\}$ is weakly reduced with respect to $f\in K\{x\}$ if no proper derivative of $v_f$ appears in $g$; if in addition the degree of $v_f$ in $g$ is strictly less than $d_f$ we say that $g$ is reduced with respect to $f$. A set $\Lambda\subseteq R\{x\}$ is said to be autoreduced if for any two distinct elements of $\Lambda$ one is reduced with respect to the other. Autoreduced sets are always finite, and we always write them in nondecreasing order by rank. The canonical orderly ranking on autoreduced sets is defined as follows: $\{g_1,\dots,g_r\}<\{f_1,\dots,f_s\}$ means that either there is $i\leq r,s$ such that $\operatorname{rk}(g_j)=\operatorname{rk}(f_j)$ for $j<i$ and $\operatorname{rk}(g_i)<\operatorname{rk}(f_i)$, or $r>s$ and $\operatorname{rk}(g_j)=\operatorname{rk}(f_j)$ for $j\leq s$. \ \par While it is not generally the case that prime differential ideals of $K\{x\}$ are finitely generated as differential ideals (though they are finitely generated as radical differential ideals), something close is true; they are determined by certain autoreduced subsets called characteristic sets. More precisely, if $P\subseteq K\{x\}$ is a prime differential ideal then a \notion{characteristic set} $\Lambda$ of $P$ is a minimal autoreduced subset of $P$ with respect to the ranking defined above. These minimal sets always exist, and determine the ideal $P$ in the sense that \begin{equation}\label{equchar}P=\{f\in K\{x\}:\; H_\Lambda^\ell f\in[\Lambda] \text{ for some $\ell\geq 0$} \}. \end{equation}The differential ideal on the right-hand-side is commonly denoted by $[\Lambda]:H_\Lambda^\infty$, where $[\Lambda]$ is the differential ideal generated by $\Lambda$ in $K\{x\}$. \end{FACT}\begin{FACT}{Fact}{reduced}\cite[Proposition 2.7]{Tressl2005}Suppose $\Lambda$ is a characteristic set of a prime differential ideal $P\subseteq K\{x\}$. If $f\neq 0$ is reduced with respect $\Lambda$, then $f$ is not in $P$. \end{FACT}\noindent Given $I\subseteq \U\{x\}$, let $\V(I)$ denote the zeroes (as differential solutions) of the elements of $I$ in $\U^n$. For a characteristic set $\Lambda$ of a prime differential ideal $P\subseteq K\{x\}$, the description \eqref{equchar} implies $$\V(P)\setminus \V(H_\Lambda)=\V(\Lambda)\setminus \V(H_\Lambda)$$ A consequence of Fact \ref{reduced} is that $H_\Lambda\notin P$, and hence the above equality says that $\V(P)$ and $\V(\Lambda)$ agree off a proper Kolchin-closed subset, namely $\V(H_\Lambda)$. \ \par \medskip \ \par We will need a bit more notation. We let $K\{x\}_{\leq r}$ denote the set of differential polynomials over $K$ of order at most $r$. On the other hand, letting $(x^{\xi}_i:(\xi,i)\in\N^m\times\{1,\dots,n\})$ be a collection of new variables, we set $$K\{x\}_{\leq r}^{\pol}=K[x^\xi _i:(\xi,i)\in \Gamma_n(r)].$$ More generally, if $\SS$ is a set of differential polynomials in $K\{x\}_{\leq r}$, we set $$\SS^{\pol}=\{f^{\pol}\in K\{x\}_{\leq r}^{\pol}:f\in \SS\}$$ where $f^{\pol}$ denotes the polynomial obtained by replacing the variables $\delta^\xi x_i$ in $f$ for the algebraic variables $x^\xi_i$. We also let $\V_r(\SS)$ denote the (algebraic) zero set of $\SS^{pol}$ in $\U^{\alpha(n,r)}$ where recall that $\alpha(n,r)=|\Gamma_n(r)|$. \begin{Remark}\label{onjets}If $V$ is an affine differential variety defined by the radical differential ideal $I\subseteq K\{x\}$, then for each $r$ the jet $\jet_r V$ has defining ideal given by $(I\cap K\{x\}_{\leq r})^{\pol}$. In other words, $$\jet_r V=\V_r(I\cap K\{x\}_{\leq r}).$$ \end{Remark}\ \par We can now recall the uniform companion theory $\UC_m$ of differential fields of characteristic zero with $m$ commuting derivations. For any set $\SS\subseteq K\{x\}$ we let $\SS^{(r)}$ denote the set of all $\delta^\xi f$ of order at most $r$ with $f\in \SS$. \begin{FACT}{Definition}{UCtheory}\cite{Tressl2005}A differential field $K$ is a model of $\UC_m$ if the following condition is satisfied: for every characteristic set $\Lambda$ of a prime differential ideal of $K\{x\}$, if $\V_r(\Lambda^{(r)})\setminus\V_r(H_\Lambda)\subseteq \U^{\alpha(n,r)}$ has a smooth $K$-point for some $r$ with $\Lambda\subseteq K\{x\}_{\leq r}$, then the differential variety $$\V(\Lambda)\setminus\V(H_\Lambda)\subseteq \U^{n}$$ has a (differential) $K$-point. \end{FACT}\begin{FACT}{Remark}{axiomrosenfeld}The fact that the class of differential fields that satisfy the above condition is first-order axiomatizable in the language of differential rings is the content of \cite[\S4]{Tressl2005}. The proof there relies heavily on Rosenfeld's Lemma which gives an algebraic characterization of characteristic sets of prime differential ideals \cite[Chapter IV, \S9]{Kolchi1973}. In Section \ref{axiomsuc} below we present an alternative (algebraic-geometric) axiomatization. \end{FACT}Next we prove two properties of characteristic sets of prime differential ideals that seem to be well known but to our knowledge are not explicitly stated elsewhere. \begin{FACT}{Lemma}{charsets1}Let $\Lambda$ be a characteristic set of a prime differential ideal $P\subseteq K\{x\}$. If $\Lambda\subseteq K\{x\}_{\leq r}$, then $$P\cap K\{x\}_{\leq r}=(\Lambda^{(r)}):H_\Lambda^\infty$$ where $(\Lambda^{(r)})$ denotes the ideal generated by $\Lambda^{(r)}$ in $K\{x\}_{\leq r}$. \end{FACT}\begin{proof}Since $P=[\Lambda]:H_\Lambda^\infty$, the containment $\supseteq$ is clear. Now let $f\in P\cap K\{x\}_{\leq r}$. Then $f$ is weakly reduced with respect to $\Lambda^{(r)}$. By the differential division algorithm, there is $g$ reduced with respect to $\Lambda^{(r)}$ and $\ell$ such that $$H_\Lambda^\ell \; f-g\in (\Lambda^{(r)}).$$ But then, as $f$ is in $P$, we get that $g\in [\Lambda]:H_\Lambda^\infty$. So, as $g$ is also reduced with respect to $\Lambda$, Fact~\ref{reduced} implies that $g=0$, and hence $f\in (\Lambda^{(r)}):H_\Lambda^\infty$. \end{proof}\begin{FACT}{Remark}{usefulfactjets}Suppose that $V$ is a $K$-irreducible affine differential variety with corresponding prime differential ideal $P\subseteq K\{x\}$. Putting together Remark~\ref{onjets} and Lemma~\ref{charsets1}, we obtain that if $\Lambda$ is a characteristic set of $P$ and $\Lambda\subseteq K\{x\}_{\leq r}$, then the defining ideal of $\Jet_r V$ in $K\{x\}_{\leq r}^{\pol}$ is $((\Lambda^{(r)}):H_\Lambda^{\infty})^{\pol}$. As a result, $$\Jet_r V\setminus \V_r(H_\Lambda)=\V_r(\Lambda^{(r)})\setminus\V_r(H_\Lambda).$$ \end{FACT}Recall that a field $F$ is \notion{existentially closed} (e.c. for short) in a field extension $L$ if every algebraic variety over $F$ with an $L$-point contains a $F$-point. The notion of e.c. for differential fields is defined similarly (in the category differential varieties). \begin{FACT}{Proposition}{charsets2}Suppose $\Lambda$ is a characteristic set of a prime differential ideal $P\subseteq K\{x\}$ and assume that $\Lambda\subseteq K\{x\}_{\leq r}$. Let $S=K\{x\}/P$ and $h=H_\Lambda/P\in S$, and let $$R=K\{x\}_{\leq r}/(\Lambda^{(r)}):H_\Lambda^\infty.$$ Then, $S_h$ is polynomial algebra over $R_h$. Consequently, $\operatorname{Frac}(R)$ is e.c. in $\operatorname{Frac}(S)$ as fields. \end{FACT}\begin{proof}Let $\Theta(x)^{>r}$ denote the set of derivatives $\delta^\xi x_i$ of order strictly larger than $r$. We thus have $$\Theta(x)^{>r}=\Theta_1(x)\cup \Theta_2(x)$$ where $\Theta_1(x)$ are elements of $\Theta^{>r}(x)$ that are not derivatives of any leader $v_f$ with $f\in \Lambda$, and $\Theta_2(x)=\Theta(x)^{>r}\setminus \Theta_1(x)$. We write $\bar \theta(x)$ for the coset of $\theta(x)$ in $S$. We claim that the elements of $\bar \Theta_1(x)\subseteq S$ are algebraically independent over $R_h$. Indeed, if there were $\theta_1(x),\dots,\theta_s(x)\in\Theta_1(x)$ such that $f(\bar\theta_1(x),\dots,\bar\theta_s(x))=0$ for some nonzero $f\in R_h[t_1,\dots,t_s]$, then for some $\ell$ we would get $H_\Lambda^\ell(x) f(\theta_1(x),\dots,\theta_s(x))\in P$. By the differential division algorithm, we can find $g$ reduced with respect to $\Lambda$ and $\ell'$ such that $$H_\Lambda^{\ell'}f(\theta_1(x),\dots,\theta_s(x))-g\in [\Lambda],$$ but since $\Lambda$ has order $\leq r$ and the $\theta_i(x)$'s are of order $>r$ and not a derivative a leader of $\Lambda$, we get that $g\neq 0$. But then $P$ would contain a nonzero element, namely $g$, that is reduced with respect to $\Lambda$, this contradicts Fact\ref{reduced}. \ \par We now prove that all the elements of $\bar\Theta_2(x)$ are in $R_h[\bar \Theta_1(x)]$. The result follows from this, as $S_h=R_h[\bar\Theta(x)]$. Let $\theta(x)\in \Theta_2(x)$. By the differential division algorithm, there is $g$ reduced with respect to $\Lambda$ and $\ell$ such that $H_\Lambda^\ell \theta(x)-g\in [\Lambda]$. But then, as $g\in K\{x\}_{\leq r}[\Theta_1(x)]$, we get $\bar\theta(x)\in R_h[\bar\Theta_1(x)]$. \end{proof}\ \par The above properties of characteristic sets are at the core of the proof of the Structure Theorem for differential algebras from \cite{Tressl2002}. In the next section we will use the following slightly different version of this theorem. \begin{theorem}\label{structuretheorem}Let $B$ be a differential $K$-algebra that is differentially finitely generated and a domain. Then $B_h\cong_KA_h\otimes_K P$ where $A$ is a domain and a finitely generated $K$-algebra, $h\in A$, and $P$ is a polynomial algebra over $K$. \end{theorem}\begin{proof}By the assumptions, $B$ is of the form $K\{x\}/P$ for some tuple of differential indeterminates $x=(x_1,\dots,x_n)$ and $P$ a prime differential ideal of $K\{x\}$. Let $S$, $R$, $h$ and $\bar\Theta_1$ be as in the proof of Proposition~\ref{charsets2}, with $\Lambda$ a characteristic set of $P$, then if we set $A=R$ and $P=K[\bar\theta_1]$ the proposition yields that $S_h\cong_KA_h\otimes P$ with the desired properties. \end{proof}\section{Differentially Large Fields}\labelx{sectionUCalgebraically}\noindent In the two applications of the differential Weil descent (see Sections \ref{firstapp} and \ref{secondapp}), we will use  several characterizations of $\uc_m$ given in Theorems~\ref{charac}~and~\ref{UCalgebraically} below. We first recall the definition of a large field, introduced by F. Pop in \cite{Pop1996}. \begin{FACT}{Definition}{}A field $F$ is said to be \notion{large} (or \textit{ample} in \cite[Rem. 16.12.3]{FriJar2008})if every irreducible affine algebraic variety over $F$ with a smooth $F$-point has a Zariski-dense set of $F$-points (Zariski-density is equivalent to saying that $F$ is existentially closed in $F(V)$ as fields). \end{FACT}Another equivalent formulation of largeness is that $F$ is existentially closed in the formal Laurent series field $F((t))$. Examples of large fields are pseudo algebraically closed fields, pseudo real closed fields, pseudo p-adically closed fields and the fraction field of any Henselian local ring, \cite{Pop2010}. Below, following Theorem~\ref{charac}, we will propose a notion of \notion{differentially large}. \ \par From now on a differential field will always mean a differential field in $m$ commuting derivations and of characteristic 0. \begin{Theorem}\label{charac}Let $K$ be a differential field and assume $K$ is large as a pure field. Then, the following are equivalent. \begin{enumerate}\item[(i)] $K$ is a model of $\UC_m$. \item[(ii)] If $L/K$ is a differential field extension such that $K$ is e.c. in $L$ as fields, then $K$ is e.c. in $L$ as differential fields. \item[(iii)] If $V$ is a $K$-irreducible differential variety such that for infinitely many $r\geq 0$ (equivalently: for all $r\geq 0$) $\jet_r(V)$ has a smooth $K$-point, then the set of $K$-rational points of $V$ is Kolchin dense over $K$ in $V$; in other words, for every  proper differential subvariety $W\subseteq V$ over $K$ there is a  differential $K$-point in $V\setminus W$. \end{enumerate}Further equivalent conditions may be found in \ref{UCalgebraically}, \ref{ondense}  and \ref{DiffLargeGeomAx}. \end{Theorem}\begin{proof}(i) $\Rightarrow$ (ii) Assume $V$ is a $K$-irreducible affine differential variety over $K$ with an $L$-point. We must show that it has a $K$-point. Without loss of generality, we may assume that $V$ has a Kolchin-generic $L$-point over $K$, call it $a$ (this can be achieved by replacing $V$ with the Kolchin-locus of the given $L$-point over $K$, if necessary). Let $P$ be the prime differential ideal of $K\{x\}$ defining $V$, and $\Lambda$ a characteristic set of $P$. Let $r$ be such that $\Lambda\subseteq K\{x\}_{\leq r}$. As $a$ is a Kolchin-generic point of $V$, $\nabla_r(a)$ is a Zariski-generic $L$-point of $\Jet_rV$. By Remark~\ref{usefulfactjets}, we have $$\Jet_r V\setminus \V_r(H_\Lambda)=\V_r(\Lambda^{(r)})\setminus\V_r(H_\Lambda),$$ and so the fact that $K$ is e.c. in $L$ as fields yields that $\V_s(\Lambda^{(r)})\setminus\V_s(H_\Lambda)$ has a smooth $K$-point. Thus, by definition of $\UC_m$, we get that $\V(\Lambda)\setminus\V(H_\Lambda)$ has a differential $K$-point. This point is also in $V$, as $V\setminus\V(H_\Lambda)=\V(\Lambda)\setminus\V(H_\Lambda)$. Note that this implication does \textit{not} use the largeness assumption on $K$. \ \par \medskip \ \par (ii) $\Rightarrow$ (iii) It suffices to consider the affine case. Let $L$ be the differential function field of $V$ over $K$; that is, $L$ is the fraction field of $K\{x\}/P$ where $P$ is the prime differential ideal defining $V$. By Remark~\ref{onjets}, for any $r$ we have that $\jet_r(V)=\V_r(P\cap K\{x\}_{\leq r})$. By assumption, for infinitely many values of $r$ we have that the latter has a smooth $K$-point; and by largeness of $K$, this means that $K$ is e.c. in the function field of $\V_r(P\cap K\{x\}_{\leq r})$ as fields. As $r$ can be taken to be arbitrarily large, we get that $K$ is e.c. in $L$ as fields. Hence, it is also e.c. as differential fields. As $V\setminus W$ contains an $L$-point, it also contains a $K$-point. \ \par \medskip \ \par (iii) $\Rightarrow$ (i) Let $\Lambda$ be a characteristic set of a prime differential ideal $P\subseteq K\{x\}$ such that $\V_r(\Lambda^{(r)})\setminus \V_r(H_\Lambda)$ has a smooth $K$-point for some $r$ with $\Lambda\subseteq K\{x\}_{\leq r}$. Let $V$ be the $K$-irreducible affine differential variety defined by the prime differential ideal $P$. Let $W=H_\Lambda$. We must show that $V\setminus W$ has a $K$-point (as the latter equals $\V(\Lambda)\setminus\V(H_\Lambda)$). So, it suffices to show that for all $s$ the jet $\jet_s(V)$ has a smooth $K$-point. \ \par Let $L$ be the fraction field of $K\{x\}/P$. We will first show that $K$ is e.c. in $L$ as fields. Let $F$ be the fraction field of $$R=K\{x\}_{\leq r}/(\Lambda^{(r)}):H^{\infty}.$$ Then, as $\V_r(\Lambda)\setminus\V_r(H_\Lambda)$ has a smooth $K$-point and $K$ is large, we have that $K$ is e.c. in $F$ (as fields). By Proposition~\ref{charsets2}, $(K\{x\}/P)_h$ with $h=H_\Lambda/P$ is a polynomial algebra over $R_h$, and so $F$ is e.c. in $L$. This shows that $K$ is e.c. in $L$ as fields. As $L$ contains a Kolchin-generic point of $V$, namely $a:=x/P$, for each $s\geq 0$ we have that $L$ contains a Zariski-generic point of $\jet_s V$, namely $\nabla_s(a)$. It follows that $\jet_s V$ has a smooth $K$-point for all $s$. \end{proof}\ \par We define differentially large fields as follows: \begin{FACT}{Definition}{difflarge}A differential field $K$ is said to be \notion[]{differentially large} if it is large as a field and satisfies any of the equivalent conditions of Theorem~\ref{charac}. \end{FACT}\begin{FACT}{Remark}{}The class of differentially large fields is first-order axiomatizable in the language of differential rings (with $m$ derivations). Indeed, the class of large is axiomatizable and by Remark~\ref{axiomrosenfeld} the axioms of $\uc_m$ are first-order. \end{FACT}\noindent Using Theorem \ref{structuretheorem} we get the following algebraic characterization of $\uc_m$: \begin{FACT}{Theorem}{UCalgebraically}Assume that the differential field $K$ is large as a field. The following are equivalent. \begin{enumerate}[(i)]\item $K$ is differentially large. \item For every differentially finitely generated $K$-algebra $S$ the following condition holds: \begin{enumerate}[label=THEN,leftmargin=10ex]\item[if]$S$ is a domain and $S\cong _K A\otimes _K P$ with $K$-algebras $A,P$ such that \begin{enumerate}[(a)]\item $A$ is a finitely generated $K$-algebra and a domain, \item $P$ is a polynomial algebra over $K$, and, \item there is a smooth $K$-rational point $A\lra K$, \footnote{Since $K$ is assumed to be large as a field, condition (c) is equivalent to saying that for all $h\in A\setminus \{0\}$ there is a $K$-rational point $A_h\lra K$.}\end{enumerate}\item[then]there is a differential $K$-rational point $S\lra K$. \end{enumerate}\end{enumerate}\ \par \end{FACT}\begin{proof}\noindent (i)$\Ra $(ii)Assume $K\models \uc_m$ and we know the if-condition of (ii). Because there is a $K$-rational point $A\to K$ and $K$ is large, $K$ is existentially closed in $\operatorname{Frac}(A)$. Moreover, as $S$ is a polynomial algebra over $A$, $\operatorname{Frac}(A)$ is e.c. in $\operatorname{Frac}(S)$ as fields, and hence $K$ is e.c. in $\operatorname{Frac}(S)$ as fields. By Theorem~\ref{charac}, $K$ is e.c. in $\operatorname{Frac}(S)$ also as differential fields. Thus, there is a differential $K$-rational point $S\to K$. \ \par \medskip\noindent (ii)$\Ra $(i)  Assume condition (ii). Let $\Lambda$ be a characteristic set of a prime differential ideal $P\subseteq K\{x\}$ such that $\V_r(\Lambda^{(r)})\setminus \V_r(H_\Lambda)$ has a smooth $K$-point for some $r$ with $\Lambda\subseteq K\{x\}_{\leq r}$. Let $S=(K\{x\}/P)_h$ where $h=H_\Lambda/P$, and let $$A=\left(K\{x\}_{\leq r}/(\Lambda^{(r)}):H_\Lambda^\infty\right)_h.$$ By Proposition~\ref{charsets2} and the proof of Theorem~\ref{structuretheorem}, $S\cong_K A\otimes_K P$ for a polynomial $K$-algebra $P$. The given smooth $K$-point of $\V_r(\Lambda^{(r)})\setminus \V_r(H_\Lambda)$ is a smooth $K$-rational point of $A\to K$. We thus get all the conditions in the hypothesis of (ii). This yields a differential $K$-rational point $S\to K$. This point lives in $V(\Lambda)\setminus \V(H_\Lambda)$ as desired. Note that this implication does \textit{not} use the largeness assumption. \ \par \end{proof}\ \par \noindent For existence of differentially large fields we refer to \cite[Theorem 6.2 (II)]{Tressl2005}: \begin{FACT}{Theorem}{LargeElemExtension}Let $K$ be a differential field that is large as a field. Then there is differential field extension $L$ of $K$ such that $L$ is differentially large and such that $L$ as a pure field is an elementary extension of the field $K$.\qed \end{FACT}\section{Algebraic Extensions of Differentially Large Fields and Minimal Differential Closures}\label{firstapp}\noindent We now prove that algebraic extensions of differentially large fields are again differentially large. This implies that the algebraic closure of a differentially large field is a model of $\DCF_{0,m}$. This yields new examples of differential fields with minimal differential closures. We carry on the notation and conventions from the previous section. \begin{FACT}{Theorem}{algextdifflarge}If $K$ is differentially large, then so is every algebraic extension (equipped with the induced derivations). \end{FACT}\begin{proof}Let $L/K$ be an algebraic extension. We must show that $L$ is differentially large. By Theorem \ref{charac}(ii), we may assume that $L/K$ is finite. Now let us assume the hypothesis and notation of part (ii) of Theorem \ref{UCalgebraically} with $L$ in place of $K$. We must show that there is a differential $L$-rational point $S\to L$. Consider the differential Weil descent $S'=W^{\operatorname{diff}}(S)$ from $L$ to $K$. Then, as the classical Weil descent commutes with tensors (cf. \ref{LeftAdjoint}\ref{LeftAdjointExact}), $S'\cong W(A)\otimes W(P)$ as $K$-algebras. Because Weil descent preserves smooth points (cf. \cite[Appendix 2]{Oester1984}), we get a smooth $K$-rational point $W(A)\to K$. Thus, as $K$ is differentially large, Theorem~\ref{UCalgebraically} yields a differential $K$-rational point $W^{\operatorname{diff}}(S)\to K$. By \ref{DiffWeilDescent}, we get a differential $L$-rational point $S\to L$, as desired. \end{proof}\begin{FACT}{Corollary}{mainresult}The algebraic closure of a differentially large field is differentially closed. In particular, if $K\models \operatorname{CODF}_m$, the theory of closed ordered differential fields in $m$ commuting derivations, then $K(i)\models \DCF_{0,m}$. \end{FACT}\noindent Previously known examples of differential fields with minimal differential closures are models of \CODF\ (which we denote as $\CODF_1$), see \cite{Singer1978b}, and fixed fields of models of $\DCF_{0,m}\operatorname{A}$, the theory differentially closed fields with a generic differential automorphism, see \cite{LeoSan2016}. The corollary delivers a vast variety of new differential fields with this property, namely all differentially large fields, see also \ref{LargeElemExtension}. A further application will be given in the next section. \section{Kolchin-Denseness of Rational Points in Differential Algebraic Groups}\label{secondapp}\noindent We present an application to a rationality question in differential algebraic groups. Note that if $F$ is a large field, then it follows that for any connected algebraic group $G$ over $F$ the set of $F$-rational points of $G$ is Zariski-dense. Indeed, $G$ is smooth as an algebraic variety and the identity $e\in G$ is a $K$-rational point, so largeness implies the Zariski-density. In this section, our goal is to prove \begin{Theorem}\label{ongroups}Assume $K$ is differentially large. If $G$ is a connected differential algebraic group over $K$, then the set of $K$-rational points of $G$, denoted $G(K)$, is Kolchin-dense in $G=G(\U)$. \end{Theorem}\ \par Theorem \ref{ongroups} will follow from the next proposition, which states a further characterization of differential largeness (compare with \ref{charac}(iii)). \begin{FACT}{Proposition}{ondense}Let $K$ be a differential field and assume $K$ is large as a field. Then $K$ is differentially large if and only if the following condition holds: \ \par If $V$ is a $K$-irreducible differential variety such that for infinitely many values of $r$ the jet $\jet_r V$ has a smooth $K$-point, then the set of $K$-rational points of $V$ is Kolchin-dense in $V$. \end{FACT}\begin{proof}The condition obviously implies \ref{charac}(iii). Conversely suppose $K$ is differentially large. By Corollary~\ref{mainresult}, $K^{\alg}$ is a model of $\DCF_{0,m}$, and so it suffices to show that $V(K)$ is dense in $V$ with respect to the Kolchin topology over $K^{\alg}$. \ \par \smallskip\noindent \claim $V$ is geometrically irreducible (i.e., $K^{\alg}$-irreducible) as a differential variety. \ \par \noindent \textit{Proof.}Since $V$ is $K$-irreducible as a differential variety, for all $r$ the jet $\jet_rV$ is $K$-irreducible\footnote{$K$-irreducibility in the Kolchin sense is equivalent to $K$-irreducibility in the Zariski by Kolchin's irreducibility theorem \cite[Ch. IV,\S17,Prop. 10, p. 200]{Kolchi1973}.}. By assumption, there are infinitely many values of $r$ for which $\jet_r V$ has a smooth $K$-point; and so, for all these $r$ the jet $\jet_rV$ is geometrically irreducible. Now the defining ideal of the jet $\jet_r V$ over $K^{\alg}$ is $$(I\cap K^{\alg}\{x\}_{\leq r})^{\pol}$$ where $I$ is the defining differential ideal of $V$ over $K^{\alg}$ (see Remark~\ref{onjets}). This implies that the ideal $I$ is prime, hence that $V$ is $K^{\alg}$-irreducible. \hfill$\diamond$ \ \par We now show that $V(K)$ is dense in $V$ with respect to the Kolchin topology over $K^{\alg}$. Let $Y$ be a proper differential subvariety of $V$ over $K^{\alg}$. Let $W$ be the Kolchin closure of $Y$ over $K$. Then $Y$ is a geometric component of $W$. Since $V$ is geometrically irreducible by the claim, $W$ is a proper differential subvariety of $V$. By \ref{charac}(iii), there is a $K$-point in $V\setminus W\subseteq V\setminus Y$ as required. \end{proof}\ \par \ifoldversion\OldStart \begin{FACT}{Proposition}{ondense}Assume $K$ is differentially large. If $V$ is a $K$-irreducible differential variety such that for infinitely many values of $r$ the jet $\jet_r V$ has a smooth $K$-point, then the set of $K$-rational points of $V$ is Kolchin-dense in $V$. \end{FACT}\begin{proof}By Corollary~\ref{mainresult}, $K^{\alg}$ is a model of $\DCF_{0,m}$, and so it suffices to show that $V(K)$ is dense in $V$ with respect to the Kolchin topology over $K^{\alg}$. \ \par \smallskip\noindent \Claim 1 $V$ is geometrically irreducible (i.e., $K^{\alg}$-irreducible) as a differential variety. \ \par \noindent \textit{Proof.}Since $V$ is $K$-irreducible as a differential variety, for all $r$ the jet $\jet_rV$ is $K$-irreducible\footnote{$K$-irreducible in the Kolchin sense is equivalent to $K$-irreducibility in the Zariski sense by a famous theorem Kolchin.}. By assumption, there are infinitely many values of $r$ for which $\jet_r V$ has a smooth $K$-point; and so, for all these $r$ the jet $\jet_rV$ is geometrically irreducible. Now the defining ideal of the jet $\jet_r V$ over $K^{\alg}$ is $$(I\cap K^{\alg}\{x\}_{\leq r})^{\pol}$$ where $I$ is the defining differential ideal of $V$ over $K^{\alg}$ (see Remark~\ref{onjets}). This implies that the ideal $I$ is prime, hence that $V$ is $K^{\alg}$-irreducible. \hfill$\diamond$ \ \par \smallskip\noindent \Claim 2 $V(K)$ is dense in $V$ with respect to the Kolchin topology over $K$. \ \par \noindent \textit{Proof.}Let $W$ be a proper differential subvariety of $V$ defined over $K$.  We must show that there is a $K$-point in $V\setminus W$. The existence of such a point is precisely the content of Theorem \ref{charac}(iii). \hfill$\diamond$ \ \par We now show that $V(K)$ is dense in $V$ with respect to the Kolchin topology over $K^{\alg}$. Let $Y$ be a proper differential subvariety of $V$ over $K^{\alg}$. Let $W$ be the Kolchin closure of $Y$ over $K$. Then $Y$ is a geometric component of $W$. \iflongversion\LongStart $Y$ is a geometric component of $W$ because we are able to work over $K^{\alg}$. This argument would not work over $\U$ \LongEnd\else\fi Since $V$ is geometrically irreducible by claim 1, we must have that $W$ is a proper differential subvariety of $V$, but then by claim 2 there is a $K$-point in $V\setminus W$, and so also in $V\setminus Y$. \end{proof}\OldEnd\else\fi \ \par We conclude with the proof of Theorem \ref{ongroups}. \begin{proof}[Proof of Theorem \ref{ongroups}]By Proposition \ref{ondense}, it suffices to show that for infinitely values of $r$ the jet $\jet_r G$ has a smooth $K$-rational point. By \cite[Corollary 4.2(ii)]{Pillay1997}, $G$ embeds over $K$ into a connected algebraic group $H$ defined over $K$. As we saw in Section~\ref{prelimconven}, for each $r$, $\nabla_r G$ is a differential algebraic subgroup of $\tau_r H$. As a result, $\jet_r G$ is an algebraic subgroup of $\tau_r H$, and so $\jet_r G$ is smooth. If $e$ denotes the identity of $G$, which is a $K$-point, then, for each $r$, the $K$-point $\nabla_r(e)$ is a smooth point of $\jet_rG$. \end{proof}\begin{FACT}{Remark}{}If $G$ is a connected linear algebraic group over a field $F$ of characteristic zero (with no largeness assumptions), the Unirationality Theorem implies that the $F$-rational points of $G$ are Zariski-dense. It would be interesting to study the analogous question for linear \textit{differential} algebraic groups. We have not pursued this in this note. \end{FACT}\section{Algebraic-Geometric Axiomatization of Large Differential Fields}\label{axiomsuc}\noindent In this last section we present algebraic-geometric axioms for differentially large fields in the spirit of Pierce-Pillay \cite{PiePil1998}. The presentation follows the recent algebraic-geometric axiomatization of $\DCF_{0,m}$ established in \cite{LeoSan2018}. In particular, we will use the recently developed theory of differential kernels for fields with several commuting derivations from~\cite{GusLeS2016}. One significant difference with the arguments in \cite{LeoSan2018} is that there one only requires the existence of regular realizations of differential kernels, while here we need the existence of principal realizations, see Remark~\ref{principalec} and Fact~\ref{useful}. We carry on the notation and conventions from previous sections. \ \par We will use two different orders $\leq$ and $\unlhd$ on $\N^m \times \{1,\dots,n\}$.  Given two elements $(\xi,i)$ and $(\tau,j)$ of $\N^m \times \{1,\dots,n\}$, we set $(\xi,i) \leq (\tau,j)$ if and only if $i = j$ and $\xi \leq \tau$ in the product order of $\N^m$. On the other hand, we set $(\xi,i) \unlhd (\tau,j)$ if and only if $$ (\sum \xi_k,i,\xi_1,\dots,\xi_{m}) \;\text{ is less than or equal to }\; (\sum \tau_k,j,\tau_1,\dots,\tau_m)$$ in the (left-)lexicographic order.  Note that if $x=(x_1,\ldots,x_{n})$ are differential indeterminates and we identify $(\xi,i)$ with $\delta^\xi x_i:=\delta_1^{\xi_1}\cdots\delta_m^{\xi_m}x_i$, then $\leq$ induces an order on the set of algebraic indeterminates given by $\delta^\xi x_i\leq \delta^\tau x_j$ if and only if $\delta^\tau x_j$ is a derivative of $\delta^\xi x_i$ (in particular this implies that $i=j$). On the other hand, the ordering $\unlhd$ induces the canonical orderly ranking on the set of algebraic indeterminates. \ \par We will look at field extensions of $K$ of the form \begin{equation}\label{extL}L:=K(a^\xi_i : (\xi,i) \in \Gamma_n(r))\end{equation}for some fixed $r\geq 0$. Here we use $a^\xi_i$ as a way to index the generators of $L$ over $K$. The element $(\tau,j)\in \N^m \times \{1,\dots,n\}$ is said to be a leader of $L$ if there is $\eta\in \N^m$ with $\eta\leq \tau$ and $\sum \eta_k\leq r$ such that $a^\eta_j$ is algebraic over $K(a^\xi_i : (\xi,i) \lhd (\eta,j))$, and a leader $(\tau,j)$ is a minimal leader of $L$ if there is no leader $(\xi,i)$ with $(\xi,i) < (\tau,j)$. We note that the notions of leader and minimal leader make sense even when we allow $r=\infty$. \ \par A (differential) kernel of length $r$ over $K$ is a field extension of the form $$L=K(a_i^{\xi}: (\xi,i)\in\Gamma_n(r))$$ such that there exist derivations $$D_k:K(a_i^\xi:(\xi,i)\in \Gamma_n(r-1))\to L$$ for $k=1,\dots,m$ extending $\delta_k$ and $D_ka_i^\xi=a_i^{\xi+{\bf k}}$ for all $(\xi,i)\in \Gamma_n(r-1)$, where ${\bf k}$ denotes the $m$-tuple whose $k$-th entry is one and zeroes elsewhere. \ \par Given a kernel $(L,D_1,\dots,D_k)$ of length $r$, we say that it has a prolongation of length $s\geq r$ if there is a kernel $(L',D_1',\dots,D_k')$ of length $s$ over $K$ such that $L'$ is a field extension of $L$ and each $D_k'$ extends $D_k$. We say that $(L,D_1,\dots,D_k)$ has a regular realization if there is a differential field extension $(M,\Delta'=\{\delta_1',\dots,\delta_m'\})$ of $(K,\Delta=\{\delta_1,\dots,\delta_m\})$ such that $M$ is a field extension of $L$ and $\delta_k'a_i^\xi=a_i^{\xi+{\bf k}}$ for all $(\xi,i)\in \Gamma_n(r-1)$ and $k=1,\dots ,m$. In this case we say that $g:=(a_1^{\bf 0},\dots,a_n^{\bf 0})$ is a regular realization of $L$. If in addition the minimal leaders of $L$ and those of the differential field $K\langle g\rangle$ coincide we say that $g$ is a principal realization of $L$. \begin{Remark}\label{principalec}Note that  if $g$ is a principal realization of the differential kernel $L$, then $L$ is existentially closed in $K\langle g \rangle$ as fields. Indeed, since the minimal leaders of $L$ and $K\langle g\rangle$ coincide, for every $(\xi,i)\in \N^m \times \{1,\dots,n\}$ we have that either $\delta^\xi g_i$ is in $L$ or it is algebraically independent from $K(\delta^\eta g_j:(\eta,j)\lhd (\xi,i))$. In other words, the differential ring generated by $g$ over $L$, namely $L\{ g\}$, is a polynomial ring over $L$. The claim follows. \end{Remark}\ \par In general, it is not the case that every kernel has a principal realization (not even regular). In \cite{GusLeS2016}, an upper bound $C_{r,m}^n$ was obtained for the length of a prolongation of a kernel that guarantees the existence of a principal realization. This bound depends only on the data $(r,m,n)$ and is constructed recursively as follows: $$C_{0,m}^1=0, \quad\; C_{r,m}^1=A(m-1,C_{r-1,m}^1), \quad \text{ and } \quad C_{r,m}^n=C_{C_{r,m}^{n-1},m}^1,$$ where $A(x,y)$ is the Ackermann function. For example, $$C_{r,1}^n=r, \quad\; C_{r,2}^n=2^n r \quad \text{ and }\quad C_{r,3}^1=3(2^r-1).$$ In \cite[Theorem 18]{GusLeS2016}, it is proved that \begin{FACT}{Fact}{useful}If a differential kernel $L=K(a_i^\xi:(\xi,i)\in\Gamma_n(r))$ of length $r$ has a prolongation of length $C_{r,m}^n$, then there is $r\leq h\leq C_{r,m}^n$ such that the differential kernel $K(a_i^\xi:(\xi,i)\in \Gamma_n(h))$ has a principal realization. \end{FACT}\begin{FACT}{Remark}{}Note that in the ordinary case $\Delta=\{\delta\}$ (i.e., $m=1$), we have $C_{r,1}^n=r$ by definition, and so the fact above shows that in this case every differential kernel has a principal realization (this is a classical result of Lando \cite{Lando1970}). \end{FACT}The fact above is the key to our algebraic-geometric axiomatization of differential largeness. But first we need some additional notation. For a given positive integer $n$, we let $$\alpha(n)=n\cdot\binom{C_{1,m}^n+m}{m}\quad\text{ and }\quad \beta(n)=n\cdot\binom{C_{1,m}^n-1+m}{m}.$$ We write $\pi:\U^{\alpha(n)}\to \U^{\beta(n)}$ for the projection onto the first $\beta(n)$ coordinates; i.e., setting $(x_i^{\xi})_{(\xi,i)\in\Gamma_n(C_{1,m}^n)}$ to be coordinates for $\U^{\alpha(n)}$ then $\pi$ is the map $$(x_i^{\xi})_{(\xi,i)\in\Gamma_n(C_{1,m}^n)}\mapsto (x_i^{\xi})_{(\xi,i)\in\Gamma(C_{1,m}^n-1)}.$$ It is worth noting here that $\alpha(n)=|\Gamma_n(C_{1,m}^n)|$ and $\beta(n)=|\Gamma_n(C_{1,m}^n-1)|$. We also use the projection $\psi:\U^{\alpha(n)}\to \U^{n\cdot(m+1)}$ onto the first $n\cdot(m+1)$ coordinates, that is, $$(x_i^{\xi})_{(\xi,i)\in\Gamma_n(C_{1,m}^n)}\mapsto (x_i^{\xi})_{(\xi,i)\in\Gamma_n(1)}.$$ Finally, we use the embedding $\phi:\U^{\alpha(n)}\to \U^{\beta(n)\cdot(m+1)}$ given by \begin{align*}(x_i^{\xi})_{(\xi,i)\in\Gamma_n(C_{1,m}^n)}\mapsto \biggl((x_i^{\xi})_{(\xi,i)\in\Gamma_n(C_{1,m}^n-1)},&(x_i^{\xi+{\bf 1}})_{(\xi,i)\in\Gamma_n(C_{1,m}^n-1)},\ldots\cr &\ldots,(x_i^{\xi+{\bf m}})_{(\xi,i)\in\Gamma_n(C_{1,m}^n-1)}\biggr). \end{align*}\ \par Recall from Section~\ref{prelimconven} that, given a Zariski-constructible set $X$ of $\U^n$, the first-prolongation of $X$ is denoted by $\tau X=\tau_1 X\subseteq \U^{n(m+1)}$. For the first-prolongation it is easy to give the defining equations: $\tau(X)$ is the Zariski-constructible set given by the conditions $$x\in X, \quad \text{ and }\quad  \sum_{i=1}^n\frac{\partial f_j}{\partial x_i}(x)\cdot y_{i,k} +f_j^{\delta_k}(x)=0\; \text{ for } 1\leq j\leq s, \; 1\leq k\leq m$$ where $f_1,\dots,f_s$ are generators of the ideal of polynomials over $\U$ vanishing at $X$, and each $f_j^{\delta_k}$ is obtained by applying $\delta_k$ to the coefficients of $f_j$. Note that $(a,\delta_1 a,\dots,\delta_m a)\in \tau X$ for all $a\in X$. Further, if $X$ is defined over the differential field $K$ then so is $\tau X$. \begin{Theorem}\label{DiffLargeGeomAx}Assume $K$ is a differential field that is large as a field. Then, $K$ is differentially large if and only \begin{enumerate}\item[($\diamondsuit$)] for every $K$-irreducible Zariski-closed set $W$ of $\U^{\alpha(n)}$ with a smooth $K$-point such that $\phi(W)\subseteq \tau(\pi(W))$, there is $a\in K^n$ with $(a,\delta_1 a,\dots,\delta_m a)\in \psi(W)$. \end{enumerate}\end{Theorem}\begin{proof}We will use the fact that a large and differential field $K$ is differentially large just if $K$ is existentially closed in every  differential field extension $L$ in which $K$ is existentially closed as a field (see Theorem \ref{charac}). The proof follows the strategy of \cite{LeoSan2018}, but here regular realizations are replaced by principal realizations with the appropriate adaptations. As the set up is technically somewhat intricate we give details. \ \par Assume $K$ is differentially large. Let $W$ be as in condition ($\diamondsuit$), we must find a point $a\in K^n$ such that $(a,\delta_1 a,\dots,\delta_m a)\in \psi(W)$. Let $b=(b_i^{\xi})_{(\xi,i)\in\Gamma_n(C_{1,m}^n)}$ be a Zariski-generic point of $W$ over $K$. Then $(b_i^{\xi})_{(\xi,i)\in\Gamma_n(C_{1,m}^n-1)}$ is a Zariski-generic point of $\pi(W)$ over $K$, and \begin{align*}\phi(b)&=\left((b_i^{\xi})_{(\xi,i)\in\Gamma_n(C_{1,m}^n-1)},(b_i^{\xi+{\bf 1}})_{(\xi,i)\in\Gamma_n(C_{1,m}^n-1)},\dots,(b_i^{\xi+{\bf m}})_{(\xi,i)\in\Gamma_n(C_{1,m}^n-1)}\right)\cr &\in \tau (\pi(W))\end{align*}\noindent By the standard argument for extending derivations (see \cite[Chapter~7, Theorem~5.1]{Lang2002}, for instance), there are derivations $$D'_k:K(b_i^\xi:(\xi,i)\in \Gamma_n(C_{1,m}^n-1))\to K(b_i^\xi:(\xi,i)\in\Gamma_n(C_{1,m}^n))$$ for $k=1,\dots,m$ extending $\delta_k$ and such that $D'_kb_i^\xi=b_i^{\xi+{\bf k}}$ for all $(\xi,i)\in\Gamma_n(C_{1,m}^n-1)$. Thus, $L'=K(b_i^\xi:(\xi,i)\in\Gamma_n(C_{1,m}^n))$ is a differential kernel over $K$ and, moreover, it is a prolongation of length $C_{1,m}^n$ of the differential kernel $L=K(b_i^\xi:(\xi,i)\in\Gamma_n(1))$ of length 1 with $D_k=D_k'|_{L}$. By Fact \ref{useful}, there is $r\leq h\leq C_{1,m}^n$ such that $L''=K(b_i^\xi:(\xi,i)\in\Gamma_n(h))$ has a principal realization; in particular, there is a differential field extension $(M,\Delta')$ of $(K,\Delta)$ containing $L''$ such that $\delta_k' b^{\bf 0}=b^{\bf k}$, where $b^{\bf 0}=(b_1^{\bf 0},\dots,b_n^{\bf 0})$ and similarly for $b^{\bf k}$. Then \begin{equation}\label{desire}(b^{\bf 0},\delta_1' b^{\bf 0}, \dots,\delta_m' b^{\bf 0})\in \psi(W)\end{equation}Now, since $W$ has a smooth $K$-point and $K$ is large, $K$ is e.c. in $L'$ as fields; in particular, $K$ is e.c. in $L''$ as fields. By Remark~\ref{principalec}, $L''$ is e.c. in the differential field $K\langle b^{\bf 0}\rangle$ as fields, and so $K$ is e.c. in $K\langle b^{\bf 0}\rangle$ as fields. Since $K$ is differentially large, the latter implies that $K$ is e.c. in $K\langle b^{\bf 0}\rangle$ as differential fields as well; and so, by \eqref{desire}, we can find the desired point $a$ in $K^n$. \ \par For the converse, assume $K$ is e.c. as field in a differential field extension $F$. We must show that $K$ is also e.c. in $F$ as differential field. Let $\rho(x)$ be a quantifier-free formula over $K$ (in the language of differential rings with $m$ derivations) in variables $x=(x_1,\dots,x_t)$ with a realization $c$ in $F$. We may write $$\rho(x)=\gamma(\delta^\xi x_i:(\xi,i)\in\Gamma_t(r))$$ where $\gamma((x^\xi)_{(\xi,i)\in\Gamma_t(r)})$ is a quantifier-free formula in the language of rings over $K$ for some $r$. If $r=0$, then $\rho$ is a formula in the language of rings, and so $\rho(x)$ has a realization in $K$ since $K$ is e.c. in $F$ as a field. Now assume $r>0$. Let $n:=t\cdot\binom{r-1+m}{m}$, $d:=(\delta^\xi c_i)_{(\xi,i)\in\Gamma_t(r-1)}$, and $$W:=\operatorname{Zar-loc}_K(\delta^\xi d_i:(\xi,i)\in\Gamma_n(C_{1,m}^n))\subseteq \U^{\alpha(n)}.$$ We have that $\phi(W)\subseteq \tau(\pi(W))$. Moreover, since $W$ has a smooth $F$-point (namely $(\delta^\xi d_i)_{(\xi,i)\in\Gamma_n(C_{1,m}^n)}$) and $K$ is e.c. in $F$ as fields, $W$ has a smooth $K$-point. By ($\diamondsuit$), there is $a=(a^\xi_i)_{(\xi,i)\in\Gamma_t(r-1)}\in K^n$ such that $(a,\delta_1 a,\dots,\delta_m a)\in \psi(W)$. This implies that $a^{\xi}_i=\delta^\xi a_i^{\bf 0}$ for all $(\xi,i)\in\Gamma_t(r-1)$. Thus, $$(\delta^\xi a_i^{\bf 0})_{(\xi,i)\in\Gamma_t(r)}\in\operatorname{Zar-loc}_K((\delta^\xi c_i)_{(\xi,i)\in\Gamma_t(r)})\subseteq \U^{t\cdot\binom{r+m}{m}},$$ and so, since $(\delta^\xi c_i)_{(\xi,i)\in\Gamma_t(r)}$ realizes $\gamma$, the point $(\delta^\xi a_i^{\bf 0})_{(\xi,i)\in\Gamma_t(r)}$ also realizes $\gamma$. Consequently, $K\models \rho(a^{\bf 0})$, as desired. \end{proof}\ \par It is worth noting that in the ordinary case ($m=1$) we get the values $\alpha(n)=2n$ and $\beta(n)=n$. Also, in this case, $\pi:\U^{2n}\to \U^n$ is just the projection onto the first $n$ coordinates, and $\psi, \phi:\U^{2n}\to \U^{2n}$ are both the identity map. We thus get the following \begin{Corollary}\label{ordinary}Assume that $(K,\delta)$ is an ordinary differential field of characteristic zero which is large as a field. Then, $(K,\delta)$ is differentially large if and only \begin{enumerate}\item[($\diamondsuit'$)] for every $K$-irreducible Zariski-closed set $W$ of $\U^{2n}$ with a smooth $K$-point such that $W\subseteq \tau_\delta(\pi(W))$, there is $a\in K^n$ with $(a,\delta a)\in W$. \end{enumerate}\end{Corollary}\begin{Remark} \ \begin{enumerate}[(i)]\item If $K$ is algebraically closed of characteristic 0, then Corollary~\ref{ordinary} yields the classical algebraic-geometric axiomatization of $\operatorname{DCF_0}$ given by Pierce and Pillay in \cite{PiePil1998}. \item If $K$ has a model complete theory $T$ in the language of fields and if $K$ is large, then Corollary~\ref{ordinary} yields a slight variation of the geometric axiomatization of $T_D$ given by Brouette, Cousins, Pillay and Point in \cite[Lemma 1.6]{BCPP2017}. \item For large and topological fields with a single derivation, an alternative description of differentially large fields with reference to the topology may be found in \cite{GuzRiv2006}. \end{enumerate}\end{Remark}\ \par \ifprivate \section{What next?}\begin{FACT}{:Some ad hoc questions directly following up the paper}{Question}\begin{enumerate}[(i)]\item If $K\prec L$ is an elementary extension of large fields such that $\trdeg(L/K)\geq \card(K)$ and if $d$ is a tuple of commuting derivations on $K$, can we find derivations $\partial$ on $L$ extending $d$ such that $(L,\partial)$ is differentially large? \item Let $(K,d)$ be a  differential field in $m$ commuting derivations, such that $K$ is large. Suppose $(K,d)$ is existentially closed in $(\qf K[[T_1,\ldots,T_m]],\partial)$ for all extensions $\partial$ of $d$ to $\qf K[[T_1,\ldots,T_m]]$. Is $(K,d)$ differentially large? \item If $K$ is differentially large are the constant fields existentially closed in $K$? \item (here $m=1$). Suppose $(K,d)$ is differentially large with constant field $C$ and assume that the pair of fields $(K,C)$ is $\aleph_0$-resplendent. Suppose $D\subseteq K^n$ is definable in every expansion of $(K,C)$ to a model of $\Th(K,d)$ by the same formula. Is $D$ definable in $(K,C)$? The proof for \CODF\ might go through - check! \item (here $m=1$). Let $K$ be a large subfield of $\overline{\Q}^{\alg}$. Is $K$ the constant field of a differentially large field? A variant to this question: Which differentially large subfields are there in the differential closure of $\Q$? Much weaker: Is there a proper     differentially large subfield of the differential closure of $\Q$? \end{enumerate}\end{FACT}\else\fi \ \par   \def\href#1#2{}{}\ \par \ifoldversion\OldStart \begin{FACT}{Proposition}{ondense}Assume $K$ is differentially large. If $V$ is a $K$-irreducible differential variety such that for infinitely many values of $r$ the jet $\jet_r V$ has a smooth $K$-point, then the set of $K$-rational points of $V$ is Kolchin-dense in $V$. \end{FACT}\begin{proof}By Corollary~\ref{mainresult}, $K^{\alg}$ is a model of $\DCF_{0,m}$, and so it suffices to show that $V(K)$ is dense in $V$ with respect to the Kolchin topology over $K^{\alg}$. \ \par \smallskip\noindent \Claim 1 $V$ is geometrically irreducible (i.e., $K^{\alg}$-irreducible) as a differential variety. \ \par \noindent \textit{Proof.}Since $V$ is $K$-irreducible as a differential variety, for all $r$ the jet $\jet_rV$ is $K$-irreducible\footnote{$K$-irreducible in the Kolchin sense is equivalent to $K$-irreducibility in the Zariski sense by a famous theorem Kolchin.}. By assumption, there are infinitely many values of $r$ for which $\jet_r V$ has a smooth $K$-point; and so, for all these $r$ the jet $\jet_rV$ is geometrically irreducible. Now the defining ideal of the jet $\jet_r V$ over $K^{\alg}$ is $$(I\cap K^{\alg}\{x\}_{\leq r})^{\pol}$$ where $I$ is the defining differential ideal of $V$ over $K^{\alg}$ (see Remark~\ref{onjets}). This implies that the ideal $I$ is prime, hence that $V$ is $K^{\alg}$-irreducible. \hfill$\diamond$ \ \par \smallskip\noindent \Claim 2 $V(K)$ is dense in $V$ with respect to the Kolchin topology over $K$. \ \par \noindent \textit{Proof.}Let $W$ be a proper differential subvariety of $V$ defined over $K$.  We must show that there is a $K$-point in $V\setminus W$. The existence of such a point is precisely the content of Theorem \ref{charac}(iii). \hfill$\diamond$ \ \par We now show that $V(K)$ is dense in $V$ with respect to the Kolchin topology over $K^{\alg}$. Let $Y$ be a proper differential subvariety of $V$ over $K^{\alg}$. Let $W$ be the Kolchin closure of $Y$ over $K$. Then $Y$ is a geometric component of $W$. \iflongversion\LongStart $Y$ is a geometric component of $W$ because we are able to work over $K^{\alg}$. This argument would not work over $\U$ \LongEnd\else\fi Since $V$ is geometrically irreducible by claim 1, we must have that $W$ is a proper differential subvariety of $V$, but then by claim 2 there is a $K$-point in $V\setminus W$, and so also in $V\setminus Y$. \end{proof}\OldEnd\else\fi \ \par \ifoldversion\OldStart \begin{FACT}{Proposition}{ondense}Assume $K$ is differentially large. If $V$ is a $K$-irreducible differential variety such that for infinitely many values of $r$ the jet $\jet_r V$ has a smooth $K$-point, then the set of $K$-rational points of $V$ is Kolchin-dense in $V$. \end{FACT}\begin{proof}By Corollary~\ref{mainresult}, $K^{\alg}$ is a model of $\DCF_{0,m}$, and so it suffices to show that $V(K)$ is dense in $V$ with respect to the Kolchin topology over $K^{\alg}$. \ \par \smallskip\noindent \Claim 1 $V$ is geometrically irreducible (i.e., $K^{\alg}$-irreducible) as a differential variety. \ \par \noindent \textit{Proof.}Since $V$ is $K$-irreducible as a differential variety, for all $r$ the jet $\jet_rV$ is $K$-irreducible\footnote{$K$-irreducible in the Kolchin sense is equivalent to $K$-irreducibility in the Zariski sense by a famous theorem Kolchin.}. By assumption, there are infinitely many values of $r$ for which $\jet_r V$ has a smooth $K$-point; and so, for all these $r$ the jet $\jet_rV$ is geometrically irreducible. Now the defining ideal of the jet $\jet_r V$ over $K^{\alg}$ is $$(I\cap K^{\alg}\{x\}_{\leq r})^{\pol}$$ where $I$ is the defining differential ideal of $V$ over $K^{\alg}$ (see Remark~\ref{onjets}). This implies that the ideal $I$ is prime, hence that $V$ is $K^{\alg}$-irreducible. \hfill$\diamond$ \ \par \smallskip\noindent \Claim 2 $V(K)$ is dense in $V$ with respect to the Kolchin topology over $K$. \ \par \noindent \textit{Proof.}Let $W$ be a proper differential subvariety of $V$ defined over $K$.  We must show that there is a $K$-point in $V\setminus W$. The existence of such a point is precisely the content of Theorem \ref{charac}(iii). \hfill$\diamond$ \ \par We now show that $V(K)$ is dense in $V$ with respect to the Kolchin topology over $K^{\alg}$. Let $Y$ be a proper differential subvariety of $V$ over $K^{\alg}$. Let $W$ be the Kolchin closure of $Y$ over $K$. Then $Y$ is a geometric component of $W$. \iflongversion\LongStart $Y$ is a geometric component of $W$ because we are able to work over $K^{\alg}$. This argument would not work over $\U$ \LongEnd\else\fi Since $V$ is geometrically irreducible by claim 1, we must have that $W$ is a proper differential subvariety of $V$, but then by claim 2 there is a $K$-point in $V\setminus W$, and so also in $V\setminus Y$. \end{proof}\OldEnd\else\fi \ \par \ifoldversion\OldStart \begin{FACT}{Proposition}{ondense}Assume $K$ is differentially large. If $V$ is a $K$-irreducible differential variety such that for infinitely many values of $r$ the jet $\jet_r V$ has a smooth $K$-point, then the set of $K$-rational points of $V$ is Kolchin-dense in $V$. \end{FACT}\begin{proof}By Corollary~\ref{mainresult}, $K^{\alg}$ is a model of $\DCF_{0,m}$, and so it suffices to show that $V(K)$ is dense in $V$ with respect to the Kolchin topology over $K^{\alg}$. \ \par \smallskip\noindent \Claim 1 $V$ is geometrically irreducible (i.e., $K^{\alg}$-irreducible) as a differential variety. \ \par \noindent \textit{Proof.}Since $V$ is $K$-irreducible as a differential variety, for all $r$ the jet $\jet_rV$ is $K$-irreducible\footnote{$K$-irreducible in the Kolchin sense is equivalent to $K$-irreducibility in the Zariski sense by a famous theorem Kolchin.}. By assumption, there are infinitely many values of $r$ for which $\jet_r V$ has a smooth $K$-point; and so, for all these $r$ the jet $\jet_rV$ is geometrically irreducible. Now the defining ideal of the jet $\jet_r V$ over $K^{\alg}$ is $$(I\cap K^{\alg}\{x\}_{\leq r})^{\pol}$$ where $I$ is the defining differential ideal of $V$ over $K^{\alg}$ (see Remark~\ref{onjets}). This implies that the ideal $I$ is prime, hence that $V$ is $K^{\alg}$-irreducible. \hfill$\diamond$ \ \par \smallskip\noindent \Claim 2 $V(K)$ is dense in $V$ with respect to the Kolchin topology over $K$. \ \par \noindent \textit{Proof.}Let $W$ be a proper differential subvariety of $V$ defined over $K$.  We must show that there is a $K$-point in $V\setminus W$. The existence of such a point is precisely the content of Theorem \ref{charac}(iii). \hfill$\diamond$ \ \par We now show that $V(K)$ is dense in $V$ with respect to the Kolchin topology over $K^{\alg}$. Let $Y$ be a proper differential subvariety of $V$ over $K^{\alg}$. Let $W$ be the Kolchin closure of $Y$ over $K$. Then $Y$ is a geometric component of $W$. \iflongversion\LongStart $Y$ is a geometric component of $W$ because we are able to work over $K^{\alg}$. This argument would not work over $\U$ \LongEnd\else\fi Since $V$ is geometrically irreducible by claim 1, we must have that $W$ is a proper differential subvariety of $V$, but then by claim 2 there is a $K$-point in $V\setminus W$, and so also in $V\setminus Y$. \end{proof}\OldEnd\else\fi 
\begin{thebibliography}{BCPP17}\bibitem[BCPP17]{BCPP2017}Q.~{Brouette}, G.~{Cousins}, A.~{Pillay}, and F.~{Point}. \newblock {Embedded Picard-Vessiot extensions \href{run:\DIRroot   article/DIFFALGEBRA/Real Differential Algebra/Brouette,Cousins,Pillay,Point -   Embedded Picard-Vessiot extensions.pdf}{pdf}}. \newblock {\em ArXiv e-prints}, August 2017. \ \par \bibitem[BLR90]{BoLuRa1990}Siegfried Bosch, Werner L{\"u}tkebohmert, and Michel Raynaud. \newblock {\em N\'eron models \href{run:\DIRroot   ARTICLE/CommAlg_and_AlgGeo/Bosch,Lutkebohmert,Raynaud - Neron   Models.djvu}{djvu}}, volume~21 of {\em Ergebnisse der Mathematik und ihrer   Grenzgebiete (3) [Results in Mathematics and Related Areas (3)]}. \newblock Springer-Verlag, Berlin, 1990. \ \par \bibitem[Bui94]{Buium1994}Alexandru Buium. \newblock {\em Differential algebra and {D}iophantine geometry   \href{run:\DIRroot ARTICLE/DIFFALGEBRA/Buium - Differential Algebra and   Diophantine Geometry.djvu}{djvu}}. \newblock Actualit\'es Math\'ematiques. [Current Mathematical Topics]. Hermann,   Paris, 1994. \ \par \bibitem[FJ08]{FriJar2008}Michael~D. Fried and Moshe Jarden. \newblock {\em Field arithmetic \href{run:\DIRroot   ARTICLE/CommAlg_and_AlgGeo/Algebra/Fried,Jarden - Field   Arithmetic.3rdEdition.pdf}{pdf}}, volume~11 of {\em Ergebnisse der Mathematik   und ihrer Grenzgebiete. 3. Folge. A Series of Modern Surveys in Mathematics   [Results in Mathematics and Related Areas. 3rd Series. A Series of Modern   Surveys in Mathematics]}. \newblock Springer-Verlag, Berlin, third edition, 2008. \newblock Revised by Jarden. \ \par \bibitem[GL16]{GusLeS2016}R.~{Gustavson} and O.~{Le{\'o}n S{\'a}nchez}. \newblock {Effective bounds for the consistency of differential equations}. \newblock {\em ArXiv e-prints. To appear in the Journal of Symbolic Computation   \href{run:\DIRroot article/DIFFALGEBRA/Groebner and characteristic   sets/Gustavson,Leon-Sanchez - Effective bounds for the consistency of   differential equations.pdf}{pdf}}, January 2016. \ \par \bibitem[GR06]{GuzRiv2006}Nicolas Guzy and C\'edric Rivi\`ere. \newblock Geometrical axiomatization for model complete theories of   differential topological fields \href{run:\DIRroot   article/DIFFALGEBRA/Guzy,Riviere - Geometrical axiomatization for model   complete theories of differential topological fields.pdf}{pdf}. \newblock {\em Notre Dame J. Formal Logic}, 47(3):331--341, 2006. \ \par \bibitem[Gro95]{Grothe1995a}Alexander Grothendieck. \newblock Technique de descente et th\'eor\`emes d'existence en g\'eom\'etrie   alg\'ebrique. {I}. {G}\'en\'eralit\'es. {D}escente par morphismes   fid\`element plats \href{run:\DIRroot ARTICLE/Bourbaki/Seminaire/Grothendieck   - Technique de Descente et Theoremes d'existence en Geometrie Algebrique. I.   Generalites. Descente par Morphismes Fidelement Plat.pdf}{pdf}. \newblock In {\em S\'eminaire {B}ourbaki, {V}ol.\ 5}, pages Exp.\ No.\ 190,   299--327. Soc. Math. France, Paris, 1995. \ \par \bibitem[Kol73]{Kolchi1973}E.~R. Kolchin. \newblock {\em Differential algebra and algebraic groups \href{run:\DIRroot   article/DIFFALGEBRA/Kolchin - Differential Algebra and Algebraic   Groups.djvu}{djvu}}. \newblock Academic Press, New York-London, 1973. \newblock Pure and Applied Mathematics, Vol. 54. \ \par \bibitem[Lan70]{Lando1970}Barbara~A. Lando. \newblock Jacobi's bound for the order of systems of first order differential   equations \href{run:\DIRroot article/DIFFALGEBRA/Groebner and characteristic   sets/Lando - Jacobi's bound for the order of systems of first order   differential equations.pdf}{pdf}. \newblock {\em Trans. Amer. Math. Soc.}, 152:119--135, 1970. \ \par \bibitem[Lan02]{Lang2002}Serge Lang. \newblock {\em Algebra \href{run:\DIRroot   article/CommAlg_and_AlgGeo/Algebra/Lang - Algebra (Springer, 2002, 3rd ed.   934s.djvu}{djvu}}, volume 211 of {\em Graduate Texts in Mathematics}. \newblock Springer-Verlag, New York, third edition, 2002. \ \par \bibitem[LS13]{LeoSan2013}Omar Le\'on~S\'anchez. \newblock {\em Contributions to the model theory of partial differential fields   \href{run:\DIRroot ARTICLE/DIFFALGEBRA/Leon Sanchez - Contributions to the   model theory of partial differential fields_thesis.pdf}{pdf}}. \newblock PhD thesis, 2013. \ \par \bibitem[LS16]{LeoSan2016}Omar Le\'on~S\'anchez. \newblock On the model companion of partial differential fields with an   automorphism \href{run:\DIRroot article/DIFFALGEBRA/Leon Sanchez - On the   model companion of partial differential fields with an   automorphism.pdf}{pdf}. \newblock {\em Israel J. Math.}, 212(1):419--442, 2016. \ \par \bibitem[LS18]{LeoSan2018}Omar Le\'on~S\'anchez. \newblock Algebro-geometric axioms for {DCF$_{0,m}$}. \newblock {\em Fundamenta Mathematicae}, 2018. \ \par \bibitem[LSM16]{LeSMos2016}Omar Le\'on~S\'anchez and Rahim Moosa. \newblock The model companion of differential fields with free operators   \href{run:\DIRroot article/DIFFALGEBRA/Leon-Sanchez,Moosa - The Model   Companion of Differential Fields With Free Operators.pdf}{pdf}. \newblock {\em J. Symb. Log.}, 81(2):493--509, 2016. \ \par \bibitem[Mat89]{Matsum1989}Hideyuki Matsumura. \newblock {\em Commutative ring theory \href{run:\DIRroot   ARTICLE/CommAlg_and_AlgGeo/Matsumura - Commutative ring theory   2ed.djvu}{djvu}}, volume~8 of {\em Cambridge Studies in Advanced   Mathematics}. \newblock Cambridge University Press, Cambridge, second edition, 1989. \newblock Translated from the Japanese by M. Reid. \ \par \bibitem[ML98]{MacLan1998}Saunders Mac~Lane. \newblock {\em Categories for the working mathematician \href{run:\DIRroot   ARTICLE/MODELTHEORY/category theory/Maclane - Categories for the Working   Mathematician.2nd Ed.djvu}{djvu}}, volume~5 of {\em Graduate Texts in   Mathematics}. \newblock Springer-Verlag, New York, second edition, 1998. \ \par \bibitem[MPS08]{MoPiSc2008}Rahim Moosa, Anand Pillay, and Thomas Scanlon. \newblock Differential arcs and regular types in differential fields   \href{run:\DIRroot ARTICLE/DIFFALGEBRA/Moosa,Pillay,Scanlon - Differential   arcs and regular types in differential fields.pdf}{pdf}. \newblock {\em J. Reine Angew. Math.}, 620:35--54, 2008. \ \par \bibitem[MS10]{MooSca2010}Rahim Moosa and Thomas Scanlon. \newblock Jet and prolongation spaces \href{run:\DIRroot   article/DIFFALGEBRA/Moosa,Scanlon - Jet and prolongation spaces.pdf}{pdf}. \newblock {\em J. Inst. Math. Jussieu}, 9(2):391--430, 2010. \ \par \bibitem[Oes84]{Oester1984}Joseph Oesterl\'e. \newblock Nombres de {T}amagawa et groupes unipotents en caract\'eristique   {$p$} \href{run:\DIRroot ARTICLE/CommAlg_and_AlgGeo/large, hilbertian, PAC   fields/Oesterle - Nombres de Tamagawa et groupes unipotents en   caracteristique p.pdf}{pdf}. \newblock {\em Invent. Math.}, 78(1):13--88, 1984. \ \par \bibitem[Pil97]{Pillay1997}Anand Pillay. \newblock Some foundational questions concerning differential algebraic groups   \href{run:\DIRroot ARTICLE/DIFFALGEBRA/Pillay - Some foundational questions   concerning differential algebraic groups.pdf}{pdf}. \newblock {\em Pacific J. Math.}, 179(1):179--200, 1997. \ \par \bibitem[Pop96]{Pop1996}Florian Pop. \newblock Embedding problems over large fields \href{run:\DIRroot   ARTICLE/DIFFALGEBRA/Pop/Pop - Embedding problems over Large Fields.pdf}{pdf}. \newblock {\em Ann. of Math. (2)}, 144(1):1--34, 1996. \ \par \bibitem[Pop10]{Pop2010}Florian Pop. \newblock Henselian implies large \href{run:\DIRroot   article/CommAlg_and_AlgGeo/large, hilbertian, PAC fields/Pop - Henselian   implies large.pdf}{pdf}. \newblock {\em Ann. of Math. (2)}, 172(3):2183--2195, 2010. \ \par \bibitem[PP98]{PiePil1998}David Pierce and Anand Pillay. \newblock A note on the axioms for differentially closed fields of   characteristic zero \href{run:\DIRroot article/DIFFALGEBRA/PIERCE/Pierce,   Pillay - A Note on the Axioms for Differentially Closed Fields of   Characteristic Zero.pdf}{pdf}. \newblock {\em J. Algebra}, 204(1):108--115, 1998. \ \par \bibitem[Ser94]{Serre1994}Jean-Pierre Serre. \newblock {\em Cohomologie galoisienne \href{run:\DIRroot article/LNM/LNM 0005   - Serre - Cohomologie Galoisienne.pdf}{pdf}}, volume~5 of {\em Lecture Notes   in Mathematics}. \newblock Springer-Verlag, Berlin, fifth edition, 1994. \ \par \bibitem[Sin78a]{Singer1978b}Michael~F. Singer. \newblock A class of differential fields with minimal differential closures   \href{run:\DIRroot article/DIFFALGEBRA/Singer/Singer - A Class of   Differential Fields with Minimal Differential Closures.pdf}{pdf}. \newblock {\em Proc. Amer. Math. Soc.}, 69(2):319--322, 1978. \ \par \bibitem[Sin78b]{Singer1978a}Michael~F. Singer. \newblock The model theory of ordered differential fields \href{run:\DIRroot   article/DIFFALGEBRA/Singer/Singer - The Model Theory of Ordered Differential   Fields.pdf}{pdf}. \newblock {\em J. Symbolic Logic}, 43(1):82--91, 1978. \ \par \bibitem[Tre02]{Tressl2002}Marcus Tressl. \newblock A structure theorem for differential algebras \href{run:\DIRroot   TEX/papers/published/Structure Theorem/Tressl - A Structure Theorem for   Differential Algebras.pdf}{pdf}. \newblock In {\em Differential {G}alois theory ({B}edlewo, 2001)}, volume~58 of   {\em Banach Center Publ.}, pages 201--206. Polish Acad. Sci. Inst. Math.,   Warsaw, 2002. \ \par \bibitem[Tre05]{Tressl2005}Marcus Tressl. \newblock The uniform companion for large differential fields of characteristic   0 \href{run:\DIRroot TEX/papers/published/uc/Tressl - The uniform companion   for large differential fields of characteristic 0.pdf}{pdf}. \newblock {\em Trans. Amer. Math. Soc.}, 357(10):3933--3951, 2005. \ \par \bibitem[Wei82]{Weil1982}Andr\'e Weil. \newblock {\em Adeles and algebraic groups \href{run:\DIRroot   ARTICLE/CommAlg_and_AlgGeo/Weil - Adeles and Algebraic Groups.pdf}{pdf}},   volume~23 of {\em Progress in Mathematics}. \newblock Birkh\"auser, Boston, Mass., 1982. \newblock With appendices by M. Demazure and Takashi Ono. \ \par \end{thebibliography}
\end{document}